\documentclass{amsart}
\usepackage{latexsym}
\usepackage{graphicx}
\usepackage{subfigure}
\usepackage[top=3.2cm,bottom=3.8cm,left=3cm,right=2cm]{geometry}
\usepackage{mathrsfs}
\usepackage{amssymb}
\usepackage{amsbsy}
\usepackage[colorlinks,linkcolor=blue,citecolor=blue,pagebackref]{hyperref}
\newtheorem{theorem}{Theorem}[section]
\newtheorem{lemma}[theorem]{Lemma}

\theoremstyle{definition}
\newtheorem{definition}{Definition}[section]
\newtheorem{example}[theorem]{Example}

\theoremstyle{remark}

\def \mS{\mathbb{S}^{n-1}}
\def \mR{\mathbb{R}^n}
\numberwithin{equation}{section}

\begin{document}

\begin{center}
	{\Large \bf A  matroid polytope approach  to sharp
 affine isoperimetric inequalities \\ for volume decomposition functionals}
\end{center}
\vskip 15pt
\begin{center}
	{\bf Yude\ Liu\ \  \ \ \ Qiang\ Sun\ \  \ \ \ Ge\ Xiong}\\~~ \\
	School of Mathematical Sciences, Key Laboratory of Intelligent Computing
and Applications\\(Ministry of Education), Tongji University, Shanghai, 200092, P.R. China
\end{center}

\vskip 5pt
\footnotetext{E-mail addresses: 1. 1910743@tongji.edu.cn; \ 2. 1910744@tongji.edu.cn;\ 3. xiongge@tongji.edu.cn}
\footnotetext{Research of the authors was supported by NSFC No. 12271407.}

\begin{center}
	\begin{minipage}{14cm}
		{\bf Abstract}
		New sharp affine isoperimetric inequalities for volume decomposition functionals $X_{2}$ and $X_{3}$ in $\mathbb{R}^n$ are established. To fulfil this task, we prove the recursion formulas for volume decomposition functionals and find out the connection between the domains of these functionals and matroid polytopes. Applications of matroid theory to convex geometry are presented.
		

	\end{minipage}
\end{center}


\vskip 15pt
\section{\bf Introduction}
\vskip 10pt

The setting for this paper is the $n$-dimensional Euclidean space, $\mathbb{R}^n$.
A \emph{convex body} is a compact convex set that has a nonempty interior.
Denote by $\mathcal{K}_o^n$ the set of convex bodies in $\mathbb{R}^n$ with the origin  $o$ in their interiors.
A \emph{polytope} in $\mathbb{R}^n$ is the convex hull of a finite set of points
in $\mathbb{R}^n$. A \emph{face} of a polytope $P$  is a convex subset $F\subseteq P$ such that $x,y\in P$ and $\frac{x+y}{2}\in F$ implies $x,y\in F.$
A face of dimension $\dim P-1$ is called a \emph{facet.}
Write $\mathcal{P}_o^n$ for the set of polytopes in $\mathbb{R}^{n}$ with the origin in their interiors. For $P\in\mathcal{P}_o^n$, write $V_n(P)$ for its $n$-dimensional volume.

Suppose $P\in \mathcal{P}_o^n$ and $u$  is a unit outer normal  vector to a facet $F$ of $P$. The
\emph{cone-volume} $V_P(\{u\})$ of $P$ associated with $u$ is the volume of the convex hull of  origin  $o$
and facet $F$. The simplest form of cone-volume is reduced to the  area formula of triangles in ancient geometry.

By using cone-volume and the idea of classification, Liu-Sun-Xiong \cite{lsx} introduced the so-called \emph{volume decomposition functional} of polytopes.

\begin{definition}\label{conefunctional} Suppose $P\in \mathcal{P}_o^n$ and the  unit outer normal  vectors of $P$
	are  $u_1,u_2,\dots,u_N$.  The \emph{$k$th volume decomposition functional} $X_k(P),$ $k=1,2,\ldots,n-1,n,$ is defined by
	\[X_k(P)^{n}= \sum_{\dim (\mathrm{span}\{u_{i_1},\dots,u_{i_n}\})=k}V_P(\{u_{i_1}\})V_P(\{u_{i_2}\}) \cdots V_P(\{u_{i_n}\}).\]
\end{definition}

Here $\mathrm{span}\{u_{i_1},\dots,u_{i_n}\}$  denotes the linear subspace
spanned by  normal vectors $u_{i_1},\dots,u_{i_n}$.
Obviously,  $X_k(P)^n$ is a homogeneous polynomial with degree $n$, $k=1,2,\ldots,n$;
$X_k(P)$ is centro-affine invariant, i.e., $X_k(TP)=X_k(P)$ for $T\in \rm{SL}(n)$;
and $X_k(\lambda P)=\lambda^nX_k(P)$ for $\lambda>0$.

It is  interesting that volume decomposition functionals satisfy the  regular identity
\begin{equation}\label{identity}
	V_n(P)^n=X_1(P)^n+X_2(P)^n+\cdots+X_n(P)^n;
\end{equation}
and  the $n$th volume decomposition functional
$$X_n(P)=\big(\sum_{u_{i_1}\wedge \dots\wedge u_{i_n}\neq 0}V_P(\{u_{i_1}\})V_P(\{u_{i_2}\})\cdots V_P(\{u_{i_n}\})\big)^{\frac{1}{n}}$$
is precisely  the functional $U$ introduced by  E. Lutwak, D. Yang and G. Zhang (LYZ) \cite{LYZ3} to attack
the longstanding Schneider projection problem. See \cite[Theorem 3.1]{lsx} for the proof of identity (\ref{identity}).

In 2001,   LYZ \cite{LYZ3} conjectured that if $P$ is a polytope in  $\mathbb{R}^{n}$  with its centroid at the origin, then
\begin{align}\label{e:1.6}
	\frac{X_n(P)}{V_{n}(P)}\geqslant\frac{(n!)^{\frac{1}{n}}}{n}
\end{align}
with equality if and only if $P$ is a parallelotope.
It took more than one dozen years to completely settle this conjecture. Refer to \cite{HLL},   \cite{Xiong} and \cite{Henk} for its developments.
In 2016, B$\rm {\ddot{o}}$r$\rm {\ddot{o}}$czky and Henk \cite{BH}  proved that LYZ's conjecture is also affirmative for convex bodies.

In solving the LYZ conjecture, a  ``\emph{concentration phenomenon}"
of cone-volumes was discovered: If $P$ is a  polytope in $\mathbb{R}^n$ with its centroid at the origin and the unit
outer normals of $P$ are  $u_1,u_2,\dots,u_N$, then
\begin{equation}\label{cone volume of polytopes}
	\sum_{u_{i}\in \xi}	V_{P}(\{u_{i}\})\leq \frac{\rm{dim}\xi}{n}V_{n}(P),\quad\quad  \text{for each subspace}\  \xi\subseteq \mathbb{R}^n,
\end{equation}
with equality  for a subspace  $\xi$ if and only if there exists a subspace $\xi^{'}$ complementary to $\xi$ in $\mathbb{R}^n$,
so that $\{u_{j}:u_{j} \not\in \xi\}\subseteq \xi^{'}$.  In 2013,  B{\"o}r{\"o}czky and
LYZ \cite{Jams} originally posed the subspace concentration condition (See Section \ref{sec2} for details), and  proved that it is this condition that guarantees the existence of solutions to the even logarithmic Minkowski problem.
For more applications,  see, e.g.,  \cite{BH1, BHP, BLYZ,HNS,LSX, SunXiong}.

\vskip3pt
In light of the identity (\ref{identity}) and  LYZ's conjecture,   Liu-Sun-Xiong \cite{lsx} raised the following  problem.

\noindent$\textbf{Problem X.}$  Let  $P\in \mathcal{P}_o^n$   with its centroid at the origin.
Does there exist a constant $c(n,k)$  depending on  $n$ and  $k, k\in \{1,2,\dots,n-1\},$ such that
$$\frac{X_k(P)}{V_{n}(P)}\leq c(n,k)?$$

The authors \cite{lsx} solved this problem in $\mathbb{R}^3$ and established the sharp affine isoperimetric inequalities
\[\frac{X_1(P)}{V_3(P)}\le\Big(\frac{1}{3}\Big)^\frac{2}{3},\ \ \frac{X_2(P)}{V_3(P)}\le\Big(\frac{2}{3}\Big)^\frac{1}{3},\ \ \text{and}\ \ \ \frac{X_3(P)}{V_3(P)}\ge \frac{2^{\frac{1}{3}}}{3^{\frac{2}{3}}},\]
and  equality holds in each inequality if and only if $P$ is a parallelepiped.

In this article, we further attack Problem X and obtain the following results.
\begin{theorem}\label{thm1.1}
	Let $P\in \mathcal{P}_o^n$ with its centroid at the origin and $n\geq 3$. Then
	$$\frac{X_2(P)}{V_n(P)}\leq\sqrt[n]{\binom{n}{2}\big((\frac{2}{n})^n-\frac{2}{n^n}\big)}$$
	with equality  if and only if $P$ is a parallelotope.
\end{theorem}

\begin{theorem}\label{thm1.2}
	Let $P\in \mathcal{P}_o^n$ with its centroid at the origin and $n\geq5$. Then
	$$\frac{X_3(P)}{V_n(P)}\leq\sqrt[n]{\binom{n}{3}\big((\frac{3}{n})^n-3(\frac{2}{n})^n+\frac{3}{n^n}\big)}$$
	with equality if and only if $P$ is a parallelotope.
\end{theorem}

Restricted to $\mathcal{P}_4^4$, i.e., the set of polytopes in $\mathbb{R}^4$  whose any \emph{four} outer  normal vectors (up to their antipodal normal vectors) are linear independent, we prove the following.

\begin{theorem}\label{thm1.3}
	Let $P\in \mathcal{P}_4^4$ with its centroid at the origin. Then
	\[ \frac{X_3(P)}{V_4(P)}\leq\sqrt[4]{\frac{72}{125}}\]
	with equality if and only if $\mathrm{supp}S_P\cup\mathrm{supp}S_{-P}=\{\pm u_1,\dots,\pm u_5\}$,
	and $V_P(\{\pm u_i\})=\frac{V_4(P)}{5},$ $i=1,\dots,5.$
\end{theorem}

It is known that if $P$ is a polygon in $\mathbb{R}^2$ with its centroid at the origin, then $\frac{X_2(P)}{V_2(P)}\geq \frac{\sqrt{2}}{2}$
with equality if and only if $P$ is a parallelogram. See \cite{Xiong} or \cite{LSX} for its proof. Combining  Theorem 1.2 in \cite{lsx} and Theorem \ref{thm1.1} together with the above fact, Problem X for $X_{1}$ and $X_{2}$ are fully solved. Now, with Theorem \ref{thm1.2} in hand, Problem X for $X_{3}$ is solved,  except the \emph{only} case in $\mathbb{R}^4$.

It is worth mentioning that Problem X for $X_3$ in $\mathbb{R}^4$ changes \emph{drastically}.
On one hand, \emph{parallelotopes} are in the set $\mathcal{P}_4^4$, but parallelotopes don't satisfy the equality conditions in Theorem \ref{thm1.3}; Meanwhile,
\emph{simplices} with their centroids at the origin attain the equality in Theorem \ref{thm1.3}.
However,  we illustrate that the functional $X_{3}$ does \emph{not} attain its maximum at simplices in the set $\mathcal{P}_o^4$. See
Example \ref{exm3.5} for details. So, it is a challenge to find the extremal body for $X_3$ in $\mathbb{R}^4$.

This article is organized as follows.

After making some necessary preliminaries in Section \ref{sec2}, we establish the  recursion formulas for volume decomposition functionals in Section \ref{sec3}. These recursion formulas are of the ``dimension reduction" function such that we can represent volume decomposition functionals as explicit polynomials.

In subsequent, we have to figure out the ``effective domains" of these polynomials.
Capturing the essential attribute of cone-volumes, we prove that these domains are indeed \emph{relative interiors} of a class of matroid polytopes, which are important in matroid theory.
Applications of matroid theory to convex geometry are presented in Section \ref{sec4}.

In Section \ref{sec5}, we prove Theorems \ref{thm1.1} and \ref{thm1.2}.
The proof of Theorem \ref{thm1.3} is provided in Section \ref{sec6}.

\vskip 15pt
\section{\bf Preliminaries}\label{sec2}
\vskip 5pt

Write $x\cdot y$ for the standard inner product of $x,y$ in $\mathbb{R}^n$. For $u\in \mathbb{R}^{n}\backslash\{o\}$ and $\alpha \in \mathbb{R}$,
\[H_{u,\alpha}=\{x\in \mathbb{R}^{n}:x\cdot u=\alpha\}\]
denotes
a \emph{hyperplane} in $\mathbb{R}^{n}$,
which bounds a \emph{closed halfspace} $H_{u,\alpha}^{-}=\{x\in \mathbb{R}^{n}:x\cdot u \leq \alpha\}$.

For $u\in\mathbb{S}^{n-1}=\{x\in \mathbb{R}^{n}:|x|=1\}$, let $l_u$ be the 1-dimensional subspace spanned by $u$.
Write $\mathrm{G}(n)$ for the set of subspaces of $\mathbb{R}^n$, and $\mathrm{G}_{n,k}$ for the set of $k$-dimensional subspaces of $\mathbb{R}^{n}$.

For $\mu\in\mathcal{B}(\mS)$, the set of finite Borel measures on $\mS$, let $\mathrm{supp}\mu$ denote its \emph{support set}.

Let $\mathcal{P}_c^n$ and $\mathcal{P}_{os}^n$ be the class of polytopes in $\mathbb{R}^{n}$ with centroid at the origin and
the class of  origin-symmetric  polytopes in $\mathbb{R}^{n}$, respectively.

For a polytope $P$ in $\mR$, write $\mathcal{F}_0(P)$ for the set of its vertices (i.e., $0$-dimensional faces), and $\mathcal{F}_1(P)$ for the set of  its edges (i.e., $1$-dimensional faces); write $\mathrm{relint}P$ for its relative interior and $\mathrm{relbd}P$ for its relative boundary, respectively.

Let $Z$ be a finite set of unit vectors in $\mathbb{R}^n$, and $Z\cup (-Z)=\{\pm u_1,\pm u_2,\dots,\pm u_N\}$.
$Z$ is said to be in \textit{$k$-general position},  $k\in \{1,2,\dots,n\},$ if $Z$ is not contained in a closed hemisphere of $\mathbb{S}^{n-1}$ and any $k$ elements of $\{u_1,u_2,\dots,u_N\}$ are linearly independent.

A polytope $P$ in $\mathbb{R}^n$ is said to be in  \emph{$k$-general position}, if the set of  unit outer normals of $P$ is in $k$-general position.
Write $\mathcal{P}_k^n$ for the set of polytopes in $\mathbb{R}^n$  which are in $k$-general position and contain the origin in their interiors.  See \cite[p. 12]{LSX} for more details about the definition. In \cite[pp. 26-28]{lsx}, we also proved that $\mathcal{P}_k^n$ is dense in $\mathcal{K}_o^n$ in the sense of Hausdorff metric $\delta$, for $k=1,2,\dots,n,$.

K\'{a}rolyi and Lov\'{a}sz \cite{KL} first posed the notion of \textit{general position}. In fact, the $n$-general position,
up to antipodal unit outer normals, is indeed the general position in the sense of K\'{a}rolyi and Lov\'{a}sz.

Let $K$ be a convex body in $\mR$.
If $K=K_1+\cdots+ K_l$ for suitable $K_i$ lying in subspaces $\xi_i$ of $\mR$ so that
$\mR=\xi_1\oplus\cdots\oplus \xi_l$, write $K=K_1\oplus\cdots\oplus K_l$. If $K = L \oplus M$ is only
possible with $\dim L = 0$ or $\dim M = 0$, then $K$ is called \emph{directly indecomposable}. If $K = L \oplus M$ with $\dim L>0$ and $\dim M>0$, then  $K$ is called an $(L,M)$-\emph{cylinder}.  Please refer to \cite[p. 156]{Schneider2} for details.

The following lemmas are needed in this article. See \cite[p. 11]{lsx} for details.

\begin{lemma}\label{thm5.1}
	Let $P$ and $Q$ be polytopes in $\mathbb{R}^n$ with the origin in their interiors. If  $V_P(\{\pm u\})=V_Q(\{\pm u\})$ for any $u\in \mathbb{S}^{n-1}$,
	then $X_k(P)=X_k(Q),$ $k=1,2,\dots,n.$
\end{lemma}

\begin{lemma}\label{injective}
	If $P$	is a polytope in $\mathbb{R}^n$ with its centroid at the origin, then there exists an origin-symmetric
	polytope $Q$ in $\mathbb{R}^n$  so that  $X_k(Q)=X_k(P),$ $k=1,2,\dots,n.$
\end{lemma}

Cone-volume measure is a natural extension of cone-volume of polytopes to convex bodies. For $K\in \mathcal{K}_o^n$,
its \emph{cone-volume measure} $V_{K}$ is a finite Borel measure on $\mathbb{S}^{n-1}$, defined for each Borel $\omega \subseteq \mathbb{S}^{n-1}$ by
\begin{equation}\label{cvm}
	V_{K}(\omega)=\frac{1}{n}\int_{x\in \nu_{K}^{-1}(\omega)}x\cdot \nu_{K}(x)d\mathcal{H}^{n-1}(x),
\end{equation}
where $\nu_{K}:\partial^{\prime}K \rightarrow \mathbb{S}^{n-1}$ is the Gauss map of $K$, defined on $\partial^{\prime}K$, the set of points
of $\partial K $ that have a unique outer unit normal, and $\mathcal{H}^{n-1}$ is $(n-1)$-dimensional Hausdorff measure.

Cone-volume measure has appeared in \cite{BGMN,Gardner1,GM, PW,Schneider2,ZouXiong},
and been intensively investigated in recent years. See, e.g., \cite{BH,BH1,BHP,BLYZ,Henk,HenkP,HX,LLSX,Xiong}.

Following  B{\"o}r{\"o}czky and LYZ \cite{Jams}, we present the definition of subspace concentration condition and the
celebrated B{\"o}r{\"o}czky-LYZ existence theorem  on solutions to the even logarithmic Minkowski problem, which will be used repeatedly.

\begin{definition}\label{scc}
	A finite Borel measure $\mu$ on $\mathbb{S}^{n-1}$ is said to satisfy the \emph{subspace concentration inequality} if, for every subspace $\xi$ of $\mathbb{R}^n$, so that $0<\mathrm{dim}\xi<n$,
	\begin{equation}
		\label{sci}
		\mu(\xi\cap\mathbb{S}^{n-1})\le\frac{1}{n}\mu(\mathbb{S}^{n-1})\mathrm{dim}\xi.
	\end{equation}
	The measure is said to satisfy the \emph{subspace concentration condition} if in addition to satisfying the subspace concentration inequality (\ref{sci}), whenever
	\[\mu(\xi\cap\mathbb{S}^{n-1})=\frac{1}{n}\mu(\mathbb{S}^{n-1})\mathrm{dim}\xi,\]
	for some subspace $\xi$, then there exists a subspace $\xi'$, which is complementary to $\xi$ in $\mathbb{R}^n$, so that also
	\[\mu(\xi'\cap\mathbb{S}^{n-1})=\frac{1}{n}\mu(\mathbb{S}^{n-1})\mathrm{dim}\xi',\]
	or equivalently so that $\mu$ is concentrated on $\mathbb{S}^{n-1}\cap(\xi\cup\xi')$.
\end{definition}

\begin{lemma}
	\label{jams}
	\emph{(B{\"o}r{\"o}czky-Lutwak-Yang-Zhang, \cite{Jams})} A non-zero finite even Borel measure
	on the unit sphere $\mathbb{S}^{n-1}$ is the cone-volume measure of an origin-symmetric convex body in $\mathbb{R}^n$
	if and only if it satisfies the subspace concentration condition.
\end{lemma}

\begin{lemma}\label{Henk}
	\emph{(Henk-Linke, \cite{Henk})}
	Let $P$ be a polytope in $\mathbb{R}^n$ with its centroid at the origin. Then its
	cone-volume measure $V_P$ satisfies the subspace concentration condition.
\end{lemma}

\vskip 15pt
\section{\bf Recursion formulas}\label{sec3}
\vskip5pt

In this part, we find the recursion formulas of the volume decomposition functionals $X_k$.

\begin{definition}\label{subXk}
	Suppose $P\in\mathcal{P}_o^{n}$ and  the unit outer normal  vectors  of $P$ are $u_1,u_2,\dots,u_N$.
	Let $\xi$ be a subspace of $\mathbb{R}^n$ and $k\in\{1,2,\dots,n-1,n\}.$ We define the quantity
	\[X_{k}(P;\xi)^{n}=\sum\limits_{\substack{u_{i_1}, \ldots, u_{i_n}\in \xi \\ \mathrm{dim}(\mathrm{span}\{u_{i_1},\ldots,u_{i_n}\})=k}} V_{P}\left(\left\{u_{i_1}\right\}\right) \cdots V_{P}\left(\left\{u_{i_n}\right\}\right).\]
\end{definition}
Note that if $\dim\xi<k$, then $X_{k}(P;\xi)=0$; If $\xi=\mR$, then
$X_{k}(P;\mathbb{R}^n)=X_{k}(P).$

For brevity,  write $\mu(\xi)=\mu(\xi\cap\mathbb{S}^{n-1})$,  for $\mu\in\mathcal{B}(\mathbb{S}^{n-1})$ and $\xi\in\mathrm{G}(n)$. For instance, if $\mu=V_P$, then $V_P(\xi)=V_P(\xi\cap\mS)$. It is interesting that we establish the \emph{local} version of  identity (\ref{identity}).

\begin{lemma}\label{msum}
	Let $\xi$ be a subspace of $\mR$. Then
	$$V_P(\xi)^n=X_{1}(P;\xi)^{n}+X_{2}(P;\xi)^{n}+\cdots+X_{\mathrm{dim}\xi}(P;\xi)^{n}.$$
\end{lemma}

\begin{proof}
	From the definitions of cone-volume measure and $X_k(P;\xi)$, it follows that
	\begin{align*}
		V_P(\xi)^n=&\big(\sum\limits_{\{i:u_i\in \xi\}}V_{P}\left(\left\{u_{i}\right\}\right)\big)^n=\sum\limits_{u_{i_1}, \ldots, u_{i_n}\in \xi}V_{P}\left(\left\{u_{i_1}\right\}\right) \cdots V_{P}\left(\left\{u_{i_n}\right\}\right)\\
		=&\sum\limits_{k=1}^{\mathrm{dim}\xi}\sum\limits_{\substack{u_{i_1}, \ldots, u_{i_n}\in \xi \\ \mathrm{dim}(\mathrm{span}\{u_{i_1},\ldots,u_{i_n}\})=k}} V_{P}\left(\left\{u_{i_1}\right\}\right) \cdots V_{P}\left(\left\{u_{i_n}\right\}\right)
		=\sum\limits_{k=1}^{\mathrm{dim}\xi}X_{k}(P;\xi)^{n}.
	\end{align*}
	That is, $V_P(\xi)^n=X_{1}(P;\xi)^{n}+X_{2}(P;\xi)^{n}+\cdots+X_{\mathrm{dim}\xi}(P;\xi)^{n}.$
\end{proof}

In particular, if $\xi$ is a 1-dimensional subspace, say $\xi=l_{u_i}$, then
\[ X_1(P;l_{u_i})^n=V_P(\{\pm u_i\})^n, \quad i=1,\dots,N;\]
If $\xi=\mR$, then
\[V_n(P)^n=V_n(\mR)^n=X_1(P)^n+X_2(P)^n+\cdots+X_n(P)^n. \]

Naturally, we pose the following

\noindent\textbf{Problem Y.}  Let $P$ be a polytope in $\mathbb{R}^n$  with its centroid at the origin.
Do there exist constants $c_1$ and $c_2$  depending on  $n$ and  $k, k\in \{1,2,\dots,n-1\},$ so that for $\xi\in\mathrm{G}_{n,k}$,
\begin{align*}
	X_k(P;\xi)\leq c_1V_{P}(\xi)\quad\text{and}\quad\frac{X_k(P;\xi)}{X_k(P)}\leq c_2?
\end{align*}

\vskip3pt
For $k\in\{1,\dots,n\}$, let
$$\{\xi_1^k,\xi_2^k,\dots,\xi_{m_k}^k\}=\{\mathrm{span}\{u_{i_1},\dots,u_{i_n}\}:i_1,\dots,i_n\in\{1,2,\dots,N\},\dim(\mathrm{span}\{u_{i_1},\dots,u_{i_n}\})=k\}.$$
That is, $\{\xi_1^k,\xi_2^k,\dots,\xi_{m_k}^k\}$ is the set of $k$-dimensional subspaces  spanned by unit outer normals of the polytope $P$, and it has a total of $m_k$ elements.

Putting $\xi=\xi_i^k$ in Lemma \ref{msum}, we obtain
\begin{equation}\label{dt}
	X_k(P;\xi_i^k)^n=V_P(\xi_i^k)^n-\sum\limits_{l=1}^{k-1}X_{l}(P;\xi_i^k)^{n},\quad  i=1,2,\dots,m_k,
\end{equation}
which suggests that $X_k(P;\xi_i^k)$ can be represented by ``lower order'' functionals $X_{l}(P;\xi_i^k)$, $l=1,\dots,k-1$.
\begin{lemma}\label{ksum}
	Let  $\xi_i^k\in \{\xi_1^k,\xi_2^k,\dots,\xi_{m_k}^k\}$, and $1\le l<k\le n$. Then
	\[X_l(P;\xi_i^k)^n=\sum\limits_{\{j\in\{1,\dots,m_l\}:\xi_j^l\subseteq \xi_i^k\}} X_l(P;\xi^l_j)^n.\]
\end{lemma}

\begin{proof}
	By the Definition \ref{subXk} and that $l<k$, it follows that
	\begin{align*}
		X_l(P;\xi_i^k)^n=&\sum\limits_{\substack{u_{i_1}, \ldots, u_{i_n}\in \xi_i^k\\ \mathrm{dim}(\mathrm{span}\{u_{i_1},\ldots,u_{i_n}\})=l}} V_{P}\left(\left\{u_{i_1}\right\}\right) \cdots V_{P}\left(\left\{u_{i_n}\right\}\right)\\
		=&\sum\limits_{\substack{v_1, \ldots, v_n \in \xi_i^k\cap\mathrm{supp} S_{P} \\ \mathrm{dim}(\mathrm{span}\left\{v_1, \ldots, v_n\right\} )=l}} V_{P}\left(\left\{v_1\right\}\right) \cdots V_{P}\left(\left\{v_n\right\}\right)\\
		=&\sum\limits_{\substack{v_1, \ldots, v_n \in \xi_i^k\cap\mathrm{supp} S_{P} \\ \operatorname{span}\left\{v_1, \ldots, v_n\right\} \in \{\xi_1^l,\dots,\xi_{m_l}^l\}}}V_{P}\left(\left\{v_1\right\}\right) \cdots V_{P}\left(\left\{v_n\right\}\right)\\
		=&\sum\limits_{\{j\in\{1,\dots,m_l\}:\xi_j^l\subseteq \xi_i^k\}}\sum\limits_{\substack{v_1, \ldots, v_n \in \xi_i^k\cap\mathrm{supp} S_{P} \\ \operatorname{span}\left\{v_1, \ldots, v_n\right\}=\xi_j^l}}V_{P}\left(\left\{v_1\right\}\right) \cdots V_{P}\left(\left\{v_n\right\}\right)\\
		=&\sum\limits_{\{j\in\{1,\dots,m_l\}:\xi_j^l\subseteq \xi_i^k\}}\sum\limits_{\substack{v_1 ,\ldots, v_n \in \xi_j^l\cap\mathrm{supp} S_{P} \\ \mathrm{dim}(\mathrm{span}\left\{v_1, \ldots, v_n\right\} )=l}}V_{P}\left(\left\{v_1\right\}\right) \cdots V_{P}\left(\left\{v_n\right\}\right)\\
		=&\sum\limits_{\{j\in\{1,\dots,m_l\}:\xi_j^l\subseteq \xi_i^k\}} X_l(P;\xi^l_j)^n,
	\end{align*}
	as desired.
\end{proof}
Putting $k=n$ in Lemma \ref{ksum}, for $1\le l<n$, we obtain
\begin{equation}\label{l3.2}
	X_l(P)^n=X_l(P;\mathbb{R}^n)^n=X_l(P;\xi^l_1)^n+X_l(P;\xi^l_2)^n+\cdots+X_l(P;\xi^l_{m_l})^n.
\end{equation}

Combining Lemma \ref{msum} and Lemma \ref{ksum}, we derive the following recursion formulas.
\begin{align}
	&X_1(P;\xi_i^1)^n=V_P(\xi_i^1)^n,\quad i=1,2,\dots,m_1;\label{22}\\
	&X_1(P;\xi)^n=\sum\limits_{\{i:\xi_i^1\subseteq \xi\}}X_1(P;\xi_i^1)^n,\quad \xi\in\mathrm{G}(n);\label{11}\\
	&X_k(P;\xi^k_i)^n=V_P(\xi^k_i)^n-\sum\limits_{l=1}^{k-1}X_l(P;\xi_i^k)^n,\quad  i=1,2,\dots,m_k,\quad k=2,3,\dots,n;\label{44}\\
	&X_l(P;\xi_i^k)^n=\sum\limits_{\{j\in\{1,\dots,m_l\}:\xi_j^l\subseteq \xi_i^k\}} X_l(P;\xi^l_j)^n,\quad  i=1,2,\dots,m_k,\quad 1\le l<k\le n.\label{33}
\end{align}

These recursion formulas  have the function of ``dimension reduction". As applications, we work out the polynomial expressions of  $X_2^n$ and $X_3^n$.

\begin{example}\label{X_2}
	Suppose $P\in\mathcal{P}_{o}^n$ and $\mathrm{supp}S_P\cup\mathrm{supp}S_{-P} =\{\pm u_1,\pm u_2,\dots,\pm u_N\}$. Then
	\[X_2(P)^n=\sum\limits_{i=1}^{m_2}\big((\sum\limits_{\{j:{u_j}\in\xi^2_{i}\}}V_P(\{\pm u_j\}))^n-\sum\limits_{\{j:{u_j}\in\xi^2_{i}\}}V_P(\{\pm u_j\})^n\big).\]
	
	Indeed, by (\ref{l3.2}), (\ref{44}), (\ref{11}), (\ref{22}) and the definition of $V_P$, we have
	\begin{align*}
		&X_2(P)^n=\sum_{i=1}^{m_2}X_2(P;\xi_i^2)^n\\
		=&\sum\limits_{i=1}^{m_2}\big( V_P(\xi_i^2)^n-X_{1}(P;\xi_i^2)^{n}\big)\\
		=&\sum\limits_{i=1}^{m_2}\big( V_P(\xi_i^2)^n-\sum\limits_{\{j:\xi_j^1\subseteq \xi_i^2\}}X_1(P;\xi_j^1)^n\big)\\
		=&\sum\limits_{i=1}^{m_2}\big( V_P(\xi_i^2)^n-\sum\limits_{\{j:\xi_j^1\subseteq \xi_i^2\}}V_P(\xi_j^1)^n\big)\\
		=&\sum\limits_{i=1}^{m_2}\big((\sum\limits_{\{j:{u_j}\in\xi^2_{i}\}}V_P(\{\pm u_j\}))^n-\sum\limits_{\{j:{u_j}\in\xi^2_{i}\}}V_P(\{\pm u_j\})^n\big),
	\end{align*}
	as desired.
\end{example}

\begin{example}\label{X_3}
	Suppose $P\in\mathcal{P}_{o}^n$ and $\mathrm{supp}S_P\cup\mathrm{supp}S_{-P}=\{\pm u_1,\pm u_2,\dots,\pm u_N\}$. Then
	\[X_3(P)^n=\sum\limits_{i=1}^{m_3}\big\{ V_P(\xi_i^3)^n-\sum\limits_{\{l:\xi_l^2\subseteq \xi_i^3\}}\big( V_P(\xi_l^2)^n-\sum\limits_{\{j:u_j\in \xi_l^2\}} V_P(\{\pm u_j\})^n\big) -\sum\limits_{\{j:u_j\in \xi_i^3\}}V_P(\{\pm u_j\})^n\big\}.\]
	
	Indeed, by (\ref{l3.2}), (\ref{44}), (\ref{33}), (\ref{11}), (\ref{22}) and the definition of $V_P$, we have
	\begin{align*}
		&X_3(P)^n=\sum_{i=1}^{m_3}X_3(P;\xi_i^3)^n\\
		=&\sum\limits_{i=1}^{m_3}\big( V_P(\xi_i^3)^n-X_{2}(P;\xi_i^3)^{n}-X_{1}(P;\xi_i^3)^{n}\big)\\
		=&\sum\limits_{i=1}^{m_3}\big( V_P(\xi_i^3)^n-\sum\limits_{\{l:\xi_l^2\subseteq \xi_i^3\}}X_{2}(P;\xi_l^2)^n-\sum\limits_{\{j:\xi_j^1\subseteq \xi_i^3\}}X_{1}(P;\xi_j^1)^n\big)\\
		=&\sum\limits_{i=1}^{m_3}\big\{ V_P(\xi_i^3)^n-\sum\limits_{\{l:\xi_l^2\subseteq \xi_i^3\}}\big( V_P(\xi_l^2)^n-\sum\limits_{\{j:\xi_j^1\subseteq \xi_l^2\}} X_1(P;\xi_j^1)^n\big) -\sum\limits_{\{j:\xi_j^1\subseteq \xi_i^3\}}X_{1}(P;\xi_j^1)^n\big\}\\
		=&\sum\limits_{i=1}^{m_3}\big\{ V_P(\xi_i^3)^n-\sum\limits_{\{l:\xi_l^2\subseteq \xi_i^3\}}\big( V_P(\xi_l^2)^n-\sum\limits_{\{j:\xi_j^1\subseteq \xi_l^2\}} V_P(\xi_j^1)^n\big) -\sum\limits_{\{j:\xi_j^1\subseteq \xi_i^3\}}V_P(\xi_j^1)^n\big\}\\
		=&\sum\limits_{i=1}^{m_3}\big\{ V_P(\xi_i^3)^n-\sum\limits_{\{l:\xi_l^2\subseteq \xi_i^3\}}\big( V_P(\xi_l^2)^n-\sum\limits_{\{j:u_j\in \xi_l^2\}} V_P(\{\pm u_j\})^n\big) -\sum\limits_{\{j:u_j\in \xi_i^3\}}V_P(\{\pm u_j\})^n\big\},
	\end{align*}
	where $V_P(\xi_i^3)=\sum\limits_{\{j:u_j\in \xi_i^3\}} V_P(\{\pm u_j\})$, $V_P(\xi_l^2)=\sum\limits_{\{j:u_j\in \xi_l^2\}} V_P(\{\pm u_j\})$, as desired.
\end{example}

In the following, we evaluate the $X_3$ of a cylinder in $\mathbb{R}^4$.
\begin{example}\label{exm3.5}
	Let $u_i=(\cos\frac{2\pi i}{6},\sin\frac{2\pi i}{6},0,0),\  i=1,2,3;$  $u_{i}=(0,0,\cos\frac{2\pi i}{8},\sin\frac{2\pi i}{8}),\  i=4,5,6,7.$
	
	Define an even discrete measure $\mu$ on $\mathbb{S}^3$ as the following
	\begin{align*}
		\mu(\{u_i\})&=\mu(\{-u_i\})=\frac{1}{12},\quad i=1,2,3;\\
		\mu(\{u_i\})&=\mu(\{-u_i\})=\frac{1}{16},\quad i=4,5,6,7.
	\end{align*}
	Then $\mu(\mathbb{S}^3)=1$, and $\mathrm{supp}\mu=\{\pm u_1,\dots,\pm u_7\}$.
	
	The set $\{\xi_1^3,\xi_2^3,\dots,\xi_{m_3}^3\}$ consists of 7  subspaces with dimension 3 as follows
	\[\mathrm{span}\{u_4,u_5,u_6,u_7,u_i\},\ i=1,2,3;\quad \quad \mathrm{span}\{u_1,u_2,u_3,u_i\},\ i=4,5,6,7,\]
	and their mass are
	\begin{align*}
		\mu(\mathrm{span}\{u_4,u_5,u_6,u_7,u_i\})&=8\times\frac{1}{16}+2\times\frac{1}{12}=\frac{2}{3}<\frac{3}{4},\quad i=1,2,3;\\
		\mu(\mathrm{span}\{u_1,u_2,u_3,u_i\})&=6\times\frac{1}{12}+2\times\frac{1}{16}=\frac{5}{8}<\frac{3}{4},\quad i=4,5,6,7.
	\end{align*}
	
	The set $\{\xi_1^2,\xi_2^2,\dots,\xi_{m_2}^2\}$ consists of 14 subspaces with dimension $2$ as follows
	\[\mathrm{span}\{u_1,u_2,u_3\};\  \mathrm{span}\{u_4,u_5,u_6,u_7\};\ \mathrm{span}\{u_i,u_j\},\ i=1,2,3,\ j=4,5,6,7,\]
	and their mass are
	\begin{align*}
		\mu(\mathrm{span}\{u_1,u_2,u_3\})=6\times\frac{1}{12}=\frac{1}{2};\ \mu(\mathrm{span}\{u_4,u_5,u_6,u_7\})=8\times\frac{1}{16}=\frac{1}{2};\\ \mu(\mathrm{span}\{u_i,u_j\})=2\times\frac{1}{12}+2\times\frac{1}{16}=\frac{7}{24}<\frac{2}{4},\quad i=1,2,3,\ j=4,5,6,7.
	\end{align*}
	
	Thus, $\mu$ satisfies subspace concentration inequalities. Since  subspace $\mathrm{span}\{u_1,u_2,u_3\}$ is complementary to subspace $\mathrm{span}\{u_4,u_5,u_6,u_7\}$ in $\mathbb{R}^4$,
	it follows that $\mu$ satisfies subspace concentration condition. By the  B{\"o}r{\"o}czky-LYZ existence theorem \ref{jams}, there exists a cylinder $P\in\mathcal{P}_{os}^4$, so that $V_P=\mu$.
	
	Putting the above digits in Example \ref{X_3}, we obtain
	\begin{align*} X_3(P)^4=&3\big\{(\frac{2}{3})^4-\big[\big((\frac{1}{2})^4-4(\frac{1}{8})^4\big)+4\big((\frac{7}{24})^4-(\frac{1}{6})^4-(\frac{1}{8})^4\big)\big]-\big[4(\frac{1}{8})^4+(\frac{1}{6})^4\big] \big\}\\ +&4\big\{(\frac{5}{8})^4-\big[\big((\frac{1}{2})^4-3(\frac{1}{6})^4\big)+3\big((\frac{7}{24})^4-(\frac{1}{6})^4-(\frac{1}{8})^4\big)\big]-\big[3(\frac{1}{6})^4+(\frac{1}{8})^4\big] \big\}\\
		\approx&0.613.
	\end{align*}
\end{example}

\noindent \textbf{Remark.} In light of Theorem \ref{thm1.3} and the fact that $0.613>0.576=\frac{72}{125}$, it yields that
\[\sup_{P\in\mathcal{P}_4^4}\frac{X_3(P)}{V_4(P)}=0.576<0.613\le\sup_{P\in\mathcal{P}_c^4}\frac{X_3(P)}{V_4(P)},\]
which suggests us that $\frac{X_3}{V_4}$ does \emph{not} attain its extremum at \emph{simplex} in $\mathcal{P}_c^4$, i.e., the set of polytopes in $\mathbb{R}^4$ with centroid at the origin.

\vskip3pt
Since $X_k(P)^n$ is a homogenous polynomial in $V_P(\{\pm u_1\}),V_P(\{\pm u_2\}),\dots,V_P(\{\pm u_N\})$, to attack the Problem X, we must characterize the domain of functional $X_k$
\[\{(V_P(\{\pm u_1\}),\dots,V_P(\{\pm u_N\}))\in\mathbb{R}^N:P\in\mathcal{P}_c^n,\ \mathrm{supp}S_P\cup\mathrm{supp}S_{-P}=\{\pm u_1,\dots,\pm u_N\}\}.\]
Moreover, since $\frac{X_k}{V_n}$ is affine invariant, we focus on the normalized domain
\begin{align}
\label{calD}&D(u_1,\dots,u_N)\\
\nonumber=\{(V_P&(\{\pm u_1\}),\dots,V_P(\{\pm u_N\}))\in\mathbb{R}^N:P\in\mathcal{P}_c^n,\ V_n(P)=n,\ \mathrm{supp}S_P\cup\mathrm{supp}S_{-P}=\{\pm u_1,\dots,\pm u_N\}\}.
\end{align}

	In Section 4, we prove that $D(u_1,\dots,u_N)$ is precisely the \emph{relative interior} of the so-called \emph{matroid polytope}. Please refer to Theorem \ref{rel=cal} for details.
	
	In the following, we list some basic facts from matroid theory. For standard reference, see \cite{GMc,Oxley}.
	\vskip15pt
	\section{\bf Matroid polytope}\label{sec4}
	
	\subsection {Introduction of matroid and matroid polytope}	\
	\vskip5pt

	Matroid theory dates from the 1930's when van der Waerden in his ``Moderne
Algebra" first approached linear and algebraic dependence axiomatically
and Whitney in his basic paper \cite{Whitney} first used the term matroid. In the past several decades, matroid theory has been developing rapidly
and witnessed strong connections with other mathematical disciplines, such as algebraic geometry \cite{ADH, FS, GGMS, HSW},  lattice theory \cite{Brylawski, GZ} and graph theory \cite{DM, Tutte}.

For a finite set $E$, write $|E|$ for the number of elements of $E$.
	\begin{definition}\label{Mdef}
		A matroid is an ordered pair $(E,\mathcal{I})$ consisting of a finite set $E$ and a
		collection $\mathcal{I} $ of subsets of $E$ with the following three properties:
		
		$\mathrm{(i)}$ $\emptyset\in \mathcal{I}$.
		
		$\mathrm{(ii)}$ If $I\in \mathcal{I}$ and  $I'\subseteq I$, then $I'\in \mathcal{I}$.
		
		$\mathrm{(iii)}$ If $I_1$ and $I_2$ are in $\mathcal{I}$ and $|I_1|<|I_2|$, then there exists an element $e\in I_2-I_1$ such that $I_1\cup e\in  \mathcal{I}$.
		
\end{definition}		

If $M$ is the matroid $(E, \mathcal{I})$, then $M$ is called a matroid on $E$.
The members of $\mathcal{I}$  are called the \emph{independent sets} of $M$; a subset of $E$ that is not in $\mathcal{I}$ is called a \emph{dependent set} of $M$. In the sense of inclusion relation of sets, we call a  \emph{maximal} independent set in $M$ a \emph{basis} of $M$,
and a  \emph{minimal} dependent set in $M$ a \emph{circuit} of $M$.
Denote all bases and all circuits of $M$ by $ \mathcal{B}(M)$ and  $ \mathcal{C}(M)$, respectively.

For $x,y\in E$,  $x$ and $y$ are \emph{equivalent} if there exists a circuit $C$ with $\{x,y\}\subseteq C.$ The equivalence classes are the \emph{connected components} of $M$. Let $c(M)$ denote the number of connected components of $M$. We say that
	$M$ is \emph{connected} if $c(M) = 1.$	

The \emph{rank} of $X\in 2^E$ is defined as
\begin{equation*}
	r_M(X)=\max\{ |I|: I\subseteq X, I\in \mathcal{I}\}.
\end{equation*}

Let $\mathrm{cl}:2^E\to2^E$, defined for all $X\in2^E$ by
$$
\mathrm{cl}(X)=\{x \in E: r_M(X \cup \{x\})=r_M(X)\}.
$$
A subset  $X$ of $E$ for which $\mathrm{cl}(X) = X$ is called a \emph{flat} or a \emph{closed set} of $M$.

%
%
%
%
%
%
%
%
%

	%

	Clearly, once $\mathcal{I}$ has
	been specified, from the definition of bases, $\mathcal{B}(M)$ is determined. Conversely, $\mathcal{I}$ can be determined
	from the set $\mathcal{B}(M)$. In fact, from  property $\mathrm{(ii)}$ in
	Definition \ref{Mdef}, we conclude that the members of $\mathcal{I}$ are precisely all subsets of
	members of $\mathcal{B}(M)$. That is,
	$$\mathcal{I}=\bigcup\limits_{B\in \mathcal{B}(M)} 2^B.$$
	 Thus a matroid $M=(E, \mathcal{I})$ is uniquely determined by the set  $\mathcal{B}(M)$.
	
	It is interesting that from  $\mathcal{B}(M)$, the so-called \emph{matroid polytope} can be defined. In fact, let $E$ be a finite set, say,  $E=\{1,\dots,N\}$.  Given a basis $B\in\mathcal{B}(M)$, the \emph{indicator vector} of $B$ is defined as $e_B=\sum\limits_{\substack{i=1\\i\in B}}^N e_i,$
	where $e_i=\{\underbrace{0,\dots,0}_{i-1},1,\underbrace{0,\dots,0}_{N-i}\}$. Then,
		the  associated matroid polytope of $M$ is given by
		$$P_M=\mathrm{conv}\{e_B: B\in \mathcal{B}(M)\}.$$
Note that by property $\mathrm{(iii)}$, it follows that $|B_1|=|B_2|$ for $B_1,B_2\in\mathcal{B}(M)$. Assume $|B|=n$ for $B\in\mathcal{B}(M)$. Then, each $e_B$ has $n$ coordinates $1$
and $(N-n)$ coordinates $0$. Refer to \cite[p. 440]{FS} and \cite[p. 311]{GGMS} for more details on matroid polytopes.



	\begin{example}
		Suppose $E=\{1,2,3,4\}$ and $\mathcal{I}=\{\emptyset, \{1\},\{2\},\{3\},\{4\},\{1,2\}, \{1,3\}, \{1,4\}, \{2,3\}, \{2,4\} \}$. Then $M=(E,\mathcal{I})$ is a matroid. From the definition of the bases, we have
		$$\mathcal{B}(M)=\{\{1,2\}, \{1,3\}, \{1,4\}, \{2,3\}, \{2,4\}\}.$$
		Then the indicator vector set associated to $\mathcal{B}(M)$ is the following set
		\begin{align*}
			\{e_1+e_2,e_1+e_3,	e_1+e_4,	e_2+e_3,	e_2+e_4\}.
		\end{align*}
	    Therefore, the  associated matroid polytope of $M$ is
	    $$P_M=\mathrm{conv}\{ (1,1,0,0), (1,0,1,0),(1,0,0,1),(0,1,1,0),(0,1,0,1)\},$$ which is a square pyramid  in the $3$-dimension space $x_1+x_2+x_3+x_4=2$.

So, $P_M\subseteq[0,1]^4\cap H_{(1,1,1,1),2}$. See  figure (a) in the blow.
	\end{example}
\vskip-15pt
\begin{figure}[h]
	\subfigure[]{
		\centering
		\label{fig1}\includegraphics[width=0.35\textwidth]{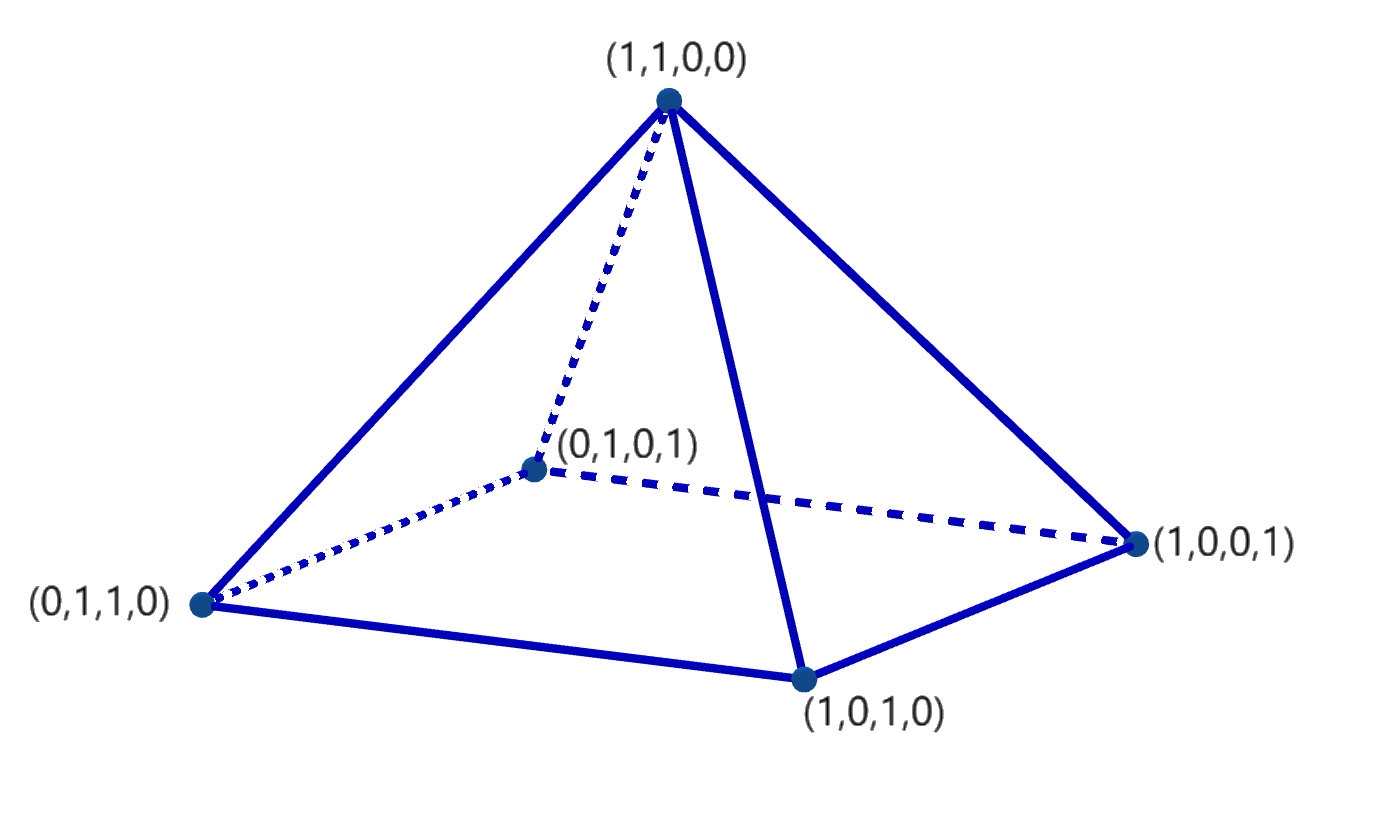}}
\subfigure[]{\centering
		\label{fig2}\includegraphics[width=0.35\textwidth]{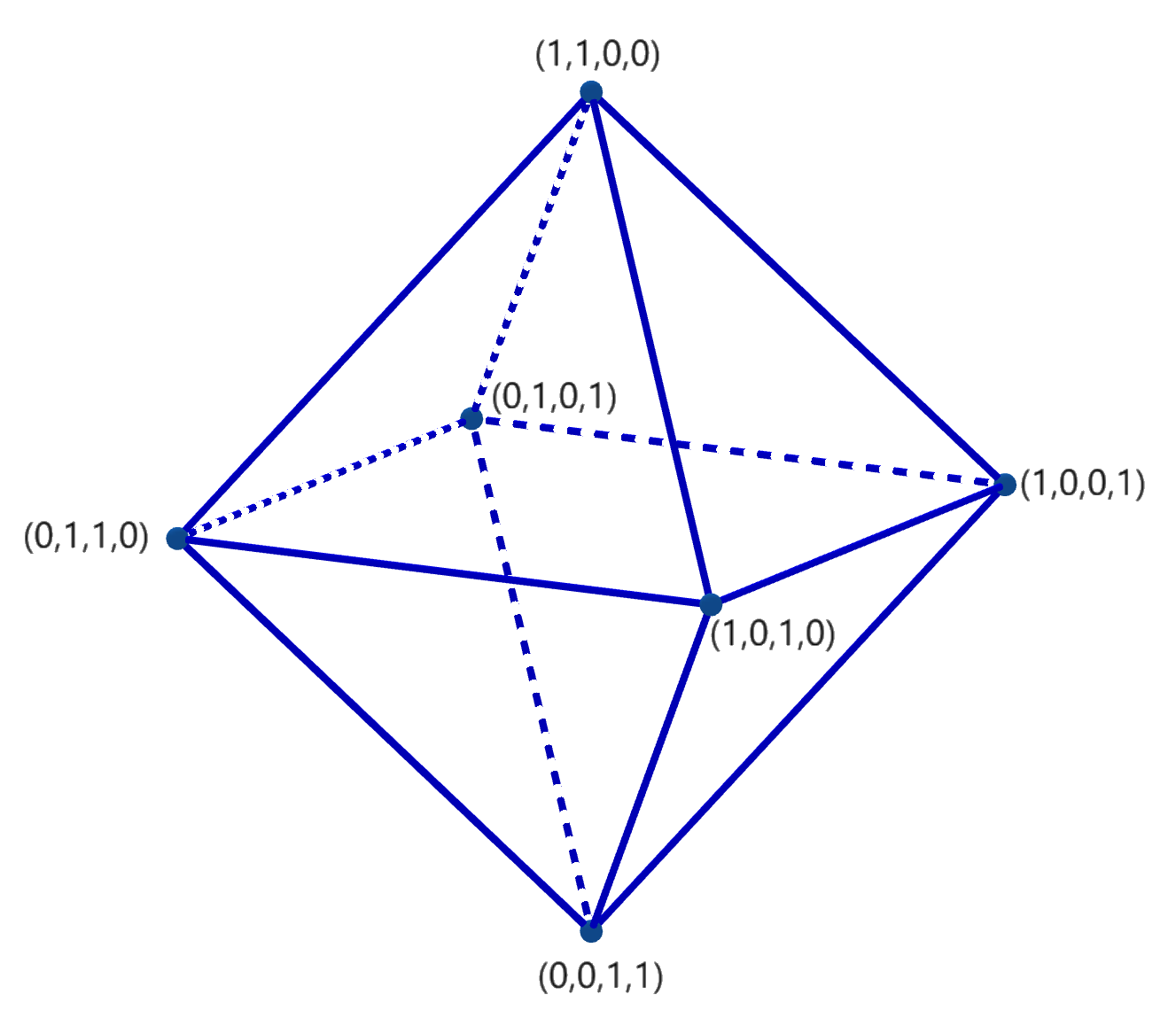}}
\end{figure}
\vskip-10pt
\begin{example}
	 If $\{3,4\}$ is also an independent set of $M$, then
	 $$\mathcal{B}(M)=\{\{1,2\}, \{1,3\}, \{1,4\}, \{2,3\}, \{2,4\},\{3,4\}\},$$
	 and  the indicator vector set associated to $\mathcal{B}(M)$ is the following set
	 \begin{align*}
	 		\{e_1+e_2,e_1+e_3,	e_1+e_4,	e_2+e_3,	e_2+e_4,e_3+e_4\}.
	 \end{align*}
	  Therefore, the  associated matroid polytope of $M$ is
	  $$P_M=\mathrm{conv}\{ (1,1,0,0), (1,0,1,0),(1,0,0,1),(0,1,1,0),(0,1,0,1),(0,0,1,1)\},$$ which is an octahedron  in the $3$-dimension space $x_1+x_2+x_3+x_4=2$.

So, $P_M=[0,1]^4\cap H_{(1,1,1,1),2}$. See figure (b) in the above.
\end{example}

	\begin{example}\label{exm3}
		In general, if  $E=\{1,\dots,N\}$ and bases $\mathcal{B}(M)$ is the collection of all $n$-element subsets of $\{ 1,\dots,N\}$, then $\mathcal{B}(M)$ has $\tbinom{N}{n}$ elements, and therefore  $P_M$ is the convex hull of $\tbinom{N}{n}$ vectors, which  have $n$ coordinates $1$
		and $(N- n)$ coordinates $0$.  Such a matroid polytope $P_M$ is called the \emph{hypersimplex} and denoted by $\bigtriangleup_N^n$, which is precisely the intersection of the $N$-cube $[0,1]^N \subseteq \mathbb{R}^N$ with the hyperplane $x_1+x_2+\cdots+x_N=n$. That is,
		$\bigtriangleup_N^n=[0,1]^N\cap H_{\mathrm{I}_N,n},$
		where $ \mathrm{I}_N=(1,\dots,1)\in\mathbb{R}^N.$
		\end{example}
		
		Significantly, there are more exact geometric information on the hypersimplex $\bigtriangleup_N^n$. For example, every edge of $\bigtriangleup_N^n$ is a translation of one of the vectors $e_i-e_j$, for $i\neq j$. Please consult \cite[p. 351]{hypers} for details. In 1987, Gelfand, Goresky, MacPherson and Serganova \cite[p. 311]{GGMS} provided an interesting geometric characterization for matroid polytope via $\bigtriangleup_N^n$.
	\begin{theorem}\emph{(Gelfand-Goresky-MacPherson-Serganova, \cite{GGMS})}
		\label{thenew1}
		Suppose $P$ is a polytope  contained in the hypersimplex $\bigtriangleup_N^n$. Then there exists a matroid $M$ such that
		$P=P_M$ if and only if the vertices of $P$ are a subset of the vertices of $\bigtriangleup_N^n$ and each edge of $P$ is a translation of one of the vectors $e_i-e_j$, for $i\neq j$.
	\end{theorem}

	In 2005, Feichtner and Sturmfels \cite[pp. 441-442]{FS} obtained an explicit representation for $P_M$ and its dimension by making use of Theorem \ref{thenew1}.

\begin{theorem}\emph{(Feichtner-Sturmfels, \cite{FS})}  \label{sturmf}Let $M=(E,\mathcal{I})$ be a matroid. Then
	\begin{equation}\label{prez0}
		P_M=\{(x_1,\dots,x_N)\in \bigtriangleup_N^n: \sum_{i\in F}x_i\le{r_M}(F)\:\text{for all flats}\:F\ \text{of}\ M\},
	\end{equation}
and its dimension $\dim P_M=n-c(M)$.
\end{theorem}

\noindent \textbf{Remark.} In our previous version of this paper submitted to a renowned jounal, we defined the so-called ``\emph{concentration polytope}" in terms of subspace concentration condition, and expended much effort to  characterize its vertices and edges. See \cite{arxiv} for details.

We  appreciate the anonymous
reviewer  very much for pointing out that our concentration polytope is precisely the matroid polytope $P_{M(u_1,\dots,u_N)}$ associated with a set of vectors $u_1,\dots,u_N$. See the following Section \ref{sec4.2} for details. More importantly, the reviewer provided two seminal papers \cite{FS} and \cite{GGMS}, and told us the vertices and edges of $P_{M(u_1,\dots,u_N)}$  have been fully characterized as Theorem \ref{thenew1}.
\vskip5pt

In the following, we return to the volume decomposition functional $X_k$ and Problem X. To solve Problem X,
we have to figure out the domain $D(u_1,\dots,u_N)$ of $X_k$. Fortunately, we discover that $D(u_1,\dots,u_N)$  is precisely the \emph{relative interior} of $P_{M(u_1,\dots,u_N)}$.
	\vskip10pt
	\subsection {The relative interior of matroid polytope}\label{sec4.2}\
	\vskip3pt
	Suppose $u_1,\dots,u_N\in\mS$ are pairwise unparallel and $\{u_1,\dots,u_N\}$  is not contained in any subsphere of $\mS$. Via these $u_i$, $i=1,\dots,N$, a matroid $M(u_1,\dots,u_N)$ is defined as the following.

	Specifically, let $E=\{1,2,\dots,N\}$ and
	$$ \mathcal{I}=\{I\in2^E:u_i,\ i\in I,\ \text{are linear independent}\}\cup \{\emptyset\}.$$
Then  $M(u_1,\dots,u_N)=(E,\mathcal{I})$ is a matroid, and
\begin{align*}
&\mathcal{B}(M)=\{B\in2^E:u_i,\ i\in B,\ \text{consist a basis of }\mR\},\\
&r_{M(u_1,\dots,u_N)}(X)=\dim (\mathrm{span} \{u_i:i\in X\}),\  \text{for }X\in 2^E,\\
&\mathrm{cl}(X)=\{i\in E: u_i\in \mathrm{span} \{u_i:i\in X\}\},\ \ \, \text{for }X\in 2^E.
\end{align*}

	By the definition of flat, $F$ is a flat of $M(u_1,\dots,u_N)$ if and only if there exists a subspace $\xi$ of $\mR$ such that $F=\{i\in E:u_i\in\xi\}$.
	Write $F_1,\dots,F_m$ for all the flats of $M(u_1,\dots,u_N)$. Let
\begin{equation}\label{defxi}
\xi_j=\mathrm{span}\{u_i:i\in F_j\}\quad  \text{and}\quad   v_j=\sum_{i\in F_j}e_i,\quad j=1,\dots,m.
\end{equation}
Then $r_{M(u_1,\dots,u_N)}(F_j)=\dim \xi_j$, and
	\begin{equation}\label{xpv}
		\sum_{i\in F_j}x_i=x\cdot v_j, \quad \forall x=(x_1,\dots,x_N)\in\mathbb{R}^N.
	\end{equation}
	Thus, by Theorem \ref{sturmf}, (\ref{xpv}) and Example \ref{exm3}, it follows that
	\begin{align}
		\nonumber P_{M(u_1,\dots,u_N)}=&\{(x_1,\dots,x_N)\in \bigtriangleup_N^n: \sum_{i\in F}x_i\le{r_{M(u_1,\dots,u_N)}}(F_j),\ j=1,\dots,m\}\\
		\nonumber=&\{(x_1,\dots,x_N)\in \bigtriangleup_N^n:x\cdot v_j\leq \dim \xi_j,\ j=1,\dots,m\}\\
		\nonumber=&[0,1]^N\cap H_{ \mathrm{I}_N,n}\cap(\cap_{j=1}^m H_{v_j,\mathrm{dim}\xi_j}^-)\\
		\label{prez}=&(\cap_{i=1}^N H_{-e_i,0}^-)\cap H_{ \mathrm{I}_N,n}\cap(\cap_{j=1}^m H_{v_j,\mathrm{dim}\xi_j}^-),
	\end{align}
where the last equality comes from that for $i\in\{1,\dots,N\}$, $\{i\}$ is indeed a flat of $M(u_1,\dots,u_N)$, say $F_i$, and therefore $v_i=e_i$ and $\dim\xi_i=\dim(\mathrm{span}\{u_i\})=1$. Consequently, $$\{(x_1,\dots,x_N)\in\mathbb{R}^N:x_i\le1\}= H_{e_i,1}^-=H_{v_i,\dim\xi_i}^-.$$

It is the representation (\ref{prez}) of $P_{M(u_1,\dots,u_N)}$ that
we fully characterize the matroid polytope $P_{M(u_1,\dots,u_N)}$ in our previous version \cite{arxiv}.

	In the following, we first characterize the relative interior of matroid polytope $P_{M(u_1,\dots,u_N)}$.
	
	Let $w_i\in\mathbb{S}^{N-1}$ and $a_i\in\mathbb{R}$, $i=1,2,\dots,M$. Suppose the set $Q=\cap_{i=1}^MH_{w_i,a_i}^-\subseteq\mathbb{R}^N$ is nonempty.
	
	\begin{lemma}\label{rintQ}
		Let $x\in Q$. Then $x\in \mathrm{relint}Q$ if and only if there does not exist an $i\in\{1,\dots,M\}$ such that $x\in H_{w_i,a_i}$ and $Q\nsubseteq H_{w_i,a_i}.$
	\end{lemma}
	
	\begin{proof}
		Assume that
		$$Q\nsubseteq H_{w_i,a_i},\  1\le i\le j;\quad \text{and}\quad Q\subseteq H_{w_i,a_i},\ j+1\le i\le M,\quad  j\in\{0,1,\dots,M\}.$$
		Then $Q$ is completely lying in an affine subspace $\bar{S}=\cap_{i=j+1}^MH_{w_i,a_i}$. If $j=0$, then $Q=\bar{S}$ is itself an affine subspace;
		If $j=M$, then $Q$ is an $N$-dimensional polytope. For these two cases, the lemma naturally holds.  So assume that $1\leq j\leq M-1$ in the following.
		
		Suppose there does not exist an $i\in\{1,\dots,M\}$
		so that
		$x\in H_{w_i,a_i},$ but $Q\nsubseteq H_{w_i,a_i}$. Then $x\in\mathrm{int}H_{w_i,a_i}^-$, $1\le i\le j$. Let $\varepsilon=\min\limits_{1\le i\le j} \{a_i-x\cdot w_i\}$. Then $\varepsilon>0$, and
		$x+\varepsilon B\subseteq H_{w_i,a_i}^-$, $1\le i\le j$. So,
		\[(x+\varepsilon B)\cap \bar{S}\subseteq (\cap_{i=1}^jH_{w_i,a_i}^-)\cap \bar{S}= (\cap_{i=1}^jH_{w_i,a_i}^-)\cap(\cap_{i=j+1}^MH_{w_i,a_i})=Q,\]
		which implies that $x\in \mathrm{relint}Q$.
		
		On the contrary, assume $x\in \mathrm{relint}Q$. We aim to show that $x\notin H_{w_i,a_i},\  1\le i\le j$. If there exists an $i\in\{1,\dots,j\}$ so that $x\in H_{w_i,a_i}$,  then there exists $y\in Q$ with $y\in \mathrm{int}H_{w_i,a_i}^-$ by that $Q\nsubseteq H_{w_i,a_i}$. From the facts that $x\in \mathrm{relint}Q$ and the  segment $[x,y]\subseteq Q$, there is a sufficiently small $\rho>0$ so that $x+\rho(x-y)\in Q$. Then
		$(x+\rho(x-y))\cdot w_i=x\cdot w_i+\rho(x-y)\cdot w_i>a_i,$
		which contradicts that $Q\subseteq H_{w_i,a_i}^-$.
	\end{proof}
%

	\begin{theorem}\label{relint}
		Let $x\in \mathbb{R}^N$.  Then $x\in\mathrm{relint}P_{M(u_1,\dots,u_N)}$ if and only if there exists an origin-symmetric polytope $P$ in $\mR$ with $V_n(P)=n$ and  $\mathrm{supp}S_P=\{\pm u_1,\dots,\pm u_N\}$, such that
		$$x=\big(V_P(\{\pm u_1\}),\dots,V_P(\{\pm u_N\})\big).$$
	\end{theorem}
	\begin{proof}
		We first prove the sufficiency.  Let $P$ be an origin-symmetric polytope with $V_n(P)=n$ and $\mathrm{supp}S_P=\{\pm u_1,\dots,\pm u_N\}$. Suppose $x=\big(V_P(\{\pm u_1\}),\dots,V_P(\{\pm u_N\})\big)$. Then
		$$x\in \mathrm{int}(H_{-e_i,0}^-),\quad i=1,\dots,N.$$
		
		 In light of formula (\ref{prez}) and Lemma \ref{rintQ}, it suffices to
		 show that  there is the following implication
		$$x\in H_{v_j,\mathrm{dim}\xi_j}\  \Longrightarrow\ P_{M(u_1,\dots,u_N)}\subseteq H_{v_j,\mathrm{dim}\xi_j},\quad j=1,\dots,m,$$
		where $\xi_j$ and $v_j$ are defined by (\ref{defxi}).

		Assume that $x\in H_{v_j,\mathrm{dim}\xi_j},$ for some $j\in\{1,\dots,m\}$. By equation (\ref{xpv}), it follows that
		\[\frac{V_P(\xi_j\cap\mathbb{S}^{n-1})}{V_n(P)}=\frac{1}{V_n(P)}\sum\limits_{\{i:u_i\in\xi_j\}}V_P(\{\pm u_i\})=\frac{1}{n}\sum\limits_{\{i:u_i\in\xi_j\}}x_i=\frac{x\cdot v_j}{n}=\frac{\mathrm{dim}\xi_j}{n}.\]
		From Lemma \ref{Henk},  there exists a subspace $\xi_j'$ complementary to $\xi_j$ such that
		\[\frac{1}{n}\sum\limits_{\{i:u_i\in\xi_j'\}}x_i=\frac{V_P(\xi_j'\cap\mathbb{S}^{n-1})}{V_n(P)}=\frac{\mathrm{dim}\xi_j'}{n}=1-\frac{\mathrm{dim}\xi_j}{n},\]
		which suggests us that $\{i:u_i\in\xi_j\}$ and $\{i:u_i\in\xi_j'\}$ constitute a disjoint partition of $\{1,2,\dots,N\}$.  So,
		for each $y=(y_1,\dots,y_N)\in P_{M(u_1,\dots,u_N)}$, it satisfies the equation
		\[\sum_{\{i:u_i\in\xi_j\}}y_i+\sum_{\{i:u_i\in\xi_j'\}}y_i=\sum_{i=1}^Ny_i=n.\]
		
		Meanwhile, by formula (\ref{prez}), it follows that
		 $$\sum_{\{i:u_i\in\xi_j\}}y_i\le\mathrm{dim}\xi_j\ \ \ \text{and}\ \ \sum_{\{i:u_i\in\xi_j'\}}y_i\le\mathrm{dim}\xi_j'=n-\mathrm{dim}\xi_j.$$
		 So, it is necessary that
		 \[\mathrm{dim}\xi_j=\sum_{\{i:u_i\in\xi_j\}}y_i=y\cdot v_j.\]
		 That is, $y\in H_{v_j,\mathrm{dim}\xi_j}.$ Therefore, $P_{M(u_1,\dots,u_N)}\subseteq H_{v_j,\mathrm{dim}\xi_j}$.
		\vskip5pt
		Second, we prove the necessity. Assume $x=(x_1,\dots,x_N)\in \mathrm{relint} P_{M(u_1,\dots,u_N)}$.
		Let $$\mu=\sum_{i=1}^{N}\frac{x_i}{2}(\delta_{u_i}+\delta_{-u_i}).$$
	 Then $\mu$ is a even discrete measure on $\mS$, and $\mu(\mS)=n$.
	
In the following, we verify that $\mu$ satisfies the subspace concentration condition.
	
	 For each proper subspace $\xi$, the subspace $\mathrm{span}(\xi\cap\{u_1,\dots,u_N\})\cup\{o\}$ is exactly one of elements of $\{\xi_1,\dots,\xi_m\}$ by the definition (\ref{defxi}), say $\xi_j$. By formula (\ref{prez}), it follows  that
\begin{equation}\label{eq4.5}
\frac{\mu(\xi\cap\mathbb{S}^{n-1})}{\mu(\mathbb{S}^{n-1})}=\frac{\sum\limits_{\{i:u_i\in\xi_j\}}x_i}{\sum_{i=1}^{N}x_i}\le\frac{\dim  \xi_j}{n}\le\frac{\dim\xi}{n}.
\end{equation}
		
		In addition, assume all the equalities hold in (\ref{eq4.5}).
		Then $\xi=\xi_j$. We aim to find a subspace $\xi_j'$
		complementary to $\xi_j$, such that $\mu(\xi'_j\cap\mathbb{S}^{n-1})=\mathrm{dim}\xi'_j$.

		First, since
		\[x\cdot v_j=\sum\limits_{\{i:u_i\in\xi_j\}}x_i=\mathrm{dim}\xi_j,\]
		it follows that $x\in H_{v_j,\mathrm{dim}\xi_j}.$ From the assumption that $x\in \mathrm{relint} P_{M(u_1,\dots,u_N)}$, and Lemma \ref{rintQ}, it follows
		that $P_{M(u_1,\dots,u_N)}\subseteq H_{v_j,\mathrm{dim}\xi_j}$.
		
		Second, since $\{u_1,\dots,u_N\}$  is not to be contained in any subsphere of $\mS$, there exists an origin-symmetric polytope $Q$ in $\mR$ with $V_n(Q)=n$ and  $\mathrm{supp}S_Q=\{\pm u_1,\dots,\pm u_N\}$.
		By the proof of the sufficient condition, we conclude that
		\begin{equation}
			\label{suppQ}
			\big( V_Q(\{\pm u_1\}),\dots,V_Q(\{\pm u_N\})\big)\in \mathrm{relint}P_{M(u_1,\dots,u_N)}\subseteq H_{v_j,\mathrm{dim}\xi_j}.
		\end{equation}
Thus, $\frac{V_Q(\xi_j\cap\mS)}{V_n(Q)}=\frac{\dim\xi_j}{n}$.
		By the Definition \ref{scc}, there exists a subspace $\xi_j'$
		complementary to $\xi_j$, such that
		\[\frac{\sum_{\{i:u_{i}\in\xi_j'\}}V_Q(\{\pm u_i\})}{V_n(Q)}=\frac{V_Q(\xi_j'\cap\mathbb{S}^{n-1})}{V_n(Q)}=\frac{\mathrm{dim}\xi_j'}{n},\]
		which suggests that $\{i:u_i\in\xi_j\}$ and $\{i:u_i\in\xi_j'\}$ constitute a disjoint partition of $\{1,2,\dots,N\}$. So,
		\[\frac{\mu(\xi'_j\cap\mathbb{S}^{n-1})}{\mu(\mS)}=\frac{1}{n}\sum_{\{i:u_i\in\xi_j'\}}x_i=\frac{1}{n}\sum_{i=1}^Nx_i-\frac{1}{n}\sum_{\{i:u_i\in\xi_j\}}x_i=1-\frac{\mathrm{dim}\xi_j}{n}=\frac{\mathrm{dim}\xi_j'}{n}.\]
		
		Therefore, $\mu$ satisfies the subspace concentration condition. By the   B{\"o}r{\"o}czky-LYZ existence theorem \ref{jams}, there exists an origin-symmetric polytope $P$ such that $V_P=\mu$. So, $V_n(P)=\mu(\mathbb{S}^{n-1})=n$ and
		\begin{equation*}
			x=(x_1,\dots,x_N)=(\mu(\{\pm u_1\}),\dots,\mu(\{\pm u_N\}))=(V_P(\{\pm u_1\}),\dots,V_P(\{\pm u_N\})).
		\end{equation*}
		
		Furthermore, in light of (\ref{suppQ}),  it follows that $P_{M(u_1,\dots,u_N)}\nsubseteq H_{-e_i,0}$, $i\in\{1,2,\dots,N\}$.  Combining this fact
		with  Lemma \ref{rintQ}, it follows that
		$x\in \mathrm{int}H_{-e_i,0}^-$. So, $V_P(\{\pm u_i\})=x_i>0$, $i=1,2,\dots,N$, which implies that
		$\mathrm{supp}S_P=\{\pm u_1,\dots,\pm u_N\}$.
		Consequently, the necessity part is derived.
	\end{proof}

With Theorem \ref{relint} in hand,  we show that the normalized domain $D(u_1,\dots,u_N)$  of the functionals $X_k$ (see the formula (\ref{calD})), is precisely the relative interior of the matroid polytope.
	\begin{theorem}\label{rel=cal}
		$D(u_1,\dots,u_N)=\mathrm{relint}P_{M(u_1,\dots,u_N)}.$
	\end{theorem}
	\begin{proof}
		By Theorem \ref{relint}, we obtain that $\mathrm{relint}P_{M(u_1,\dots,u_N)}$ is the following set
		\[\{(V_P(\{\pm u_1\}),\dots,V_P(\{\pm u_N\}))\in\mathbb{R}^N:P\in\mathcal{P}_{os}^n,\ V_n(P)=n,\ \mathrm{supp}S_P=\{\pm u_1,\dots,\pm u_N\}\}.\]
		Comparing it with the definition of $D(u_1,\dots,u_N)$, it is obvious that $\mathrm{relint}P_{M(u_1,\dots,u_N)}\subseteq D(u_1,\dots,u_N)$.
		
	On the other hand, let $P\in\mathcal{P}_c^n$ with $V_n(P)=n$ and $\mathrm{supp}S_P\cup\mathrm{supp}S_{-P}=\{\pm u_1,\dots,\pm u_N\}$. Let
		$$\mu=\sum_{i=1}^{N}\frac{V_P(\{\pm u_i\})}{2}(\delta_{u_i}+\delta_{-u_i}).$$
		Then $\mu$ is an even discrete measure on $\mS$. By Lemma \ref{Henk}, $V_P$ satisfies the subspace concentration condition, and therefore $\mu$ satisfies the subspace concentration condition. By the B{\"o}r{\"o}czky-LYZ existence theorem \ref{jams}, there is a polytope $P'\in \mathcal{P}_{os}^n$ such that $V_{P'}=\mu$. Thus, $\mathrm{supp}S_{P'}=\{\pm u_1,\dots,\pm u_N\}$ and $V_n(P')=n$, and therefore
$(V_{P'}(\{\pm u_1\}),\dots,V_{P'}(\{\pm u_N\}))\in\mathrm{relint}P_{M(u_1,\dots,u_N)}$. So,
		$$(V_P(\{\pm u_1\}),\dots,V_P(\{\pm u_N\}))=(V_{P'}(\{\pm u_1\}),\dots,V_{P'}(\{\pm u_N\}))\in\mathrm{relint}P_{M(u_1,\dots,u_N)},$$
		which yields that $D(u_1,\dots,u_N)\subseteq\mathrm{relint}P_{M(u_1,\dots,u_N)}$.
	\end{proof}
	
	Combining Theorem \ref{rel=cal} with Theorem \ref{thenew1}, the characterization for the \emph{closure} $\overline{D}(u_1,\dots,u_N)$ of $D(u_1,\dots,u_N)$ is immediately obtained.

\begin{theorem} \label{vertxedge}
	Each vertex of $\overline{D}(u_1,\dots,u_N)$  has $n$ coordinates $1$
	and $(N- n)$ coordinates $0$ and each edge of $\overline{D}(u_1,\dots,u_N)$ is a translation of one of the vectors $e_i-e_j$, for $i\neq j$.
\end{theorem}

\vskip10pt
	\subsection {Applications of matroid theory to convex geometry}\
	\vskip3pt

	
%
%
	
	Define a mapping from the matroid polytope $P_{M(u_1,\dots,u_N)}$ to the set of finite even discrete measures on $\mS$ by
	\[x\longmapsto\mu_x=\sum_{i=1}^{N}\frac{x_i}{2}(\delta_{u_i}+\delta_{-u_i}),\quad \forall x=(x_1,\dots,x_N)\in P_{M(u_1,\dots,u_N)}.\]
	Theorem \ref{relint} shows that for each $x\in\mathrm{relint}P_{M(u_1,\dots,u_N)}$, there exist an origin-symmetric  polytope $P$ such that $\mu_x=V_P$. In the following, we focus on the points on the relative boundary of $P_{M(u_1,\dots,u_N)}$.
\begin{theorem} The relative boundary of $P_{M(u_1,\dots,u_N)}$ satisfies the following properties.

$\mathrm{(i)}$ For each vertex $x$ of $P_{M(u_1,\dots,u_N)}$, there is an origin-symmetric parallelotope $P$ such that $\mu_x=V_P$.

$\mathrm{(ii)}$ For any point $x$ in the relative interior of an  edge of $P_{M(u_1,\dots,u_N)}$,  there does not exist an origin-symmetric  polytope $P$ such that $\mu_x=V_P$.

$\mathrm{(iii)}$ Let $\{u_1,\dots,u_N\}$ be in general position. For each $x\in \mathrm{relbd}P_{M(u_1,\dots,u_N)}$, there  exists an origin-symmetric  polytope $P$ such that $\mu_x=V_P$  if and only if $x$ is a vertex  or $x\cdot e_i\in[0,1)$, $i=1,\dots,N$.
\end{theorem}
\begin{proof}
	$\mathrm{(i)}$ Suppose $x$ is a vertex of $P_{M(u_1,\dots,u_N)}$. By Theorem \ref{vertxedge}, there exist $i_1,\dots,i_n\in\{1,\dots,N\}$ such that  $x=\sum_{j=1}^{n}e_{i_j}$. From the definition of matroid polytope, it follows that $\{i_1,\dots,i_n\}$ is a basis of $M(u_1,\dots,u_N)$, and therefore $\{u_{i_1},\dots,u_{i_n}\}$ is a basis of $\mR$. So, there is an origin-symmetric parallelotope $P$ with volume $n$ and unit normal vector set $\{\pm u_{i_1},\dots,\pm u_{i_n}\}.$  Hence, $\mu_x=\frac{1}{2}\sum_{j=1}^{n}(\delta_{u_{i_j}}+\delta_{-u_{i_j}})=V_P$.
	
	$\mathrm{(ii)}$ Suppose $x$ is a point in the relative interior of an  edge of $P_{M(u_1,\dots,u_N)}$. By Theorem \ref{vertxedge},
	 $x$ has $(n-1)$ coordinates $1$,  one coordinate $\lambda$ and one coordinate $(1-\lambda)$, where $\lambda\in(0,1)$. W.l.o.g. assume $x=(\underbrace{1,\dots,1}_{n-1},\lambda,1-\lambda,0,\dots,0)$. Then
\vskip-10pt
	\[\mu_x=\frac{\lambda}{2}(\delta_{u_n}+\delta_{-u_n})+\frac{1-\lambda}{2}(\delta_{u_{n+1}}+\delta_{-u_{n+1}})+\frac{1}{2}\sum_{j=1}^{n-1}(\delta_{u_j}+\delta_{-u_j}).\]
	So,
$\frac{\mu_x(\mathrm{span}\{u_1,\dots,u_{n-1}\}\cap\mS)}{\mu_x(\mS)}=\frac{n-1}{n}.$
Since $u_{n}$ is not parallel to $u_{n+1}$, it follows that  $\{\pm u_n,\pm u_{n+1}\}$ does not lie in a $1$-dimensional subspace, and therefore $\mu_x$ does not satisfy the subspace concentration condition. By Theorem \ref{jams}, there does not exist an origin-symmetric  polytope $P$ such that $\mu_x=V_P$.

	$\mathrm{(iii)}$ We first prove the sufficiency. According to $\mathrm{(i)}$, it suffices to prove the case that $x\cdot e_i\in[0,1)$, $i=1,\dots,N$. Since $u_1,\dots,u_N$ are in general position, it follows that for each proper subspace $\xi$ of $\mR$, $\xi$ contains at most $\dim\xi$ vectors among
	$u_1,\dots,u_N$. So,
	$$\frac{\mu_x(\xi\cap\mS)}{\mu_x(\mS)}=\frac{1}{n}\sum_{\{i:u_i\in\xi\}}x_i<\frac{\dim\xi}{n}.$$
	Hence, $\mu_x$ satisfies the subspace concentration condition. By Theorem \ref{jams},  there  exists an origin-symmetric  polytope $P$ such that $\mu_x=V_P$
	
	Second, we prove the necessity. Suppose that  there  exists an origin-symmetric  polytope $P$ such that $\mu_x=V_P$. It suffices to prove that if there exists an $i\in\{1,\dots,N\}$ such that $x_i=1$, then $x$ is a vertex of $P_{M(u_1,\dots,u_N)}$. Assume $x_i=1$. Then
	\[\frac{\mu_x(\pm u_i)}{\mu_x(\mS)}=\frac{1}{n}.\]
	Since $\mu_x$ satisfies the subspace concentration condition, it follows that there exists a $(n-1)$-dimensional subspace $\xi'$
	complementary to $\mathrm{span}\{u_i\}$ so that $\mathrm{supp}\mu_x\setminus\{\pm u_i\}\subseteq \xi'$. Since $u_1,\dots,u_N$ are in general position, it follows that $\xi'$ contains at most $(n-1)$ vectors among
	$u_1,\dots,u_N$. Thus, $\mathrm{supp}\mu_x$ contains precisely $n$ vectors among $u_1,\dots,u_N$. Therefore, $x$ is a vertex of $P_{M(u_1,\dots,u_N)}$ by the definition of $\mu_x$.
\end{proof}


In the following Theorem \ref{poly}, we  determine the dimension of $P_{M(u_1,\dots,u_N)}$. By the Feichtner-Sturmfels theorem \ref{sturmf}, it suffices to determine $c(M(u_1,\dots,u_N))$, i.e., the number of connected components of $M(u_1,\dots,u_N)$. Let $P$ be any origin-symmetric polytope such that $\mathrm{supp}S_P=\{\pm u_1,\dots,\pm u_N\}$ and $V_n(P)=n$. Suppose $P= P_1\oplus\cdots\oplus P_l$, $\dim P_i>0$ and $P_i$ is directly indecomposable, $i=1,\dots,l$.

\begin{theorem}\label{poly}
	$\dim P_{M(u_1,\dots,u_N)}=N-l$.
\end{theorem}
\begin{proof}
   We complete the proof by dividing two cases: $l=1$ and $l>1$.

   \textbf{Case 1.} Assume $l=1$. We aim to show that $ P_{M(u_1,\dots,u_N)}$ contains an $(N-1)$-dimensional ball. If so, then $\dim P_{M(u_1,\dots,u_N)}=N-1$ by the fact that $P_{M(u_1,\dots,u_N)}\subseteq H_{ \mathrm{I}_N,n}$.

	Indeed, by Theorem \ref{relint} we have that $x=\big((V_P(\{\pm u_1\}),\dots,V_P(\{\pm u_N\}))\in\mathrm{relint} P_{M(u_1,\dots,u_N)}$. In the following, we prove that there exists $\varepsilon>0$ such that
	$$(x+\varepsilon B)\cap H_{ \mathrm{I}_N,n}\subseteq P_{M(u_1,\dots,u_N)},$$
	where $B$ is the unit ball in $\mathbb{R}^N$.
	
	By Lemma \ref{Henk} and the assumption that $l=1$ (i.e., $P$ is not a cylinder),  it follows that
	\[ \frac{V_P(\xi\cap\mS)}{V_n(P)}<\frac{\dim \xi}{n},\quad \text{for each proper subspace}\ \xi .\]
	Let
	$$
	\varepsilon=\frac{1}{N}\min\{\dim \xi-V_P(\xi\cap\mS):\xi\in\{\xi_1,\dots,\xi_m\}\text{ is a proper subspace}\},$$
 where $\{\xi_1,\dots,\xi_m\}$ is defined by (\ref{defxi}).
	Moreover, assume $\varepsilon<V_P(\{\pm u_i\})$, $i=1,\dots,N$.
	
	For each $y\in(x+\varepsilon B)\cap H_{ \mathrm{I}_N,n}$, on one hand we have
	\[y\cdot e_i\ge x\cdot e_i-\varepsilon=V_P(\{\pm u_i\})-\varepsilon>0,\quad i=1,\dots,N.\]
Meanwhile, since $|v_j|\le N$, it follows that
	\[y\cdot v_j\leq (V_P(\{\pm u_1\}),\dots,V_P(\{\pm u_N\}))\cdot v_j+N\varepsilon=V_P(\xi_j\cap\mS)+N\varepsilon\le\dim \xi_j,\quad j=1,\dots,m.\]
	From formula (\ref{prez}), it follows that $y\in P_{M(u_1,\dots,u_N)}$, and therefore $(x+\varepsilon B)\cap H_{ \mathrm{I}_N,n}\subseteq P_{M(u_1,\dots,u_N)}.$
	
	\textbf{Case 2.} Assume $l>1$. We aim to show that $c(M(u_1,\dots,u_N))=l$.
	
	By the definition of $M(u_1,\dots,u_N)$, it follows that $M(\phi u_1,\dots,\phi u_N)=M(u_1,\dots,u_N)$ for $\phi\in \mathrm{GL}(n)$. W.l.o.g., assume $P_j\subseteq\xi_j$, $j=1,\dots,l$, where $\xi_1,\dots,\xi_l$ are pairwise orthogonal subspaces of $\mR$.
	
	Let $F_j=\{i:u_i\in\xi_j\}$, $j=1,\dots,l.$ In the following, we prove that all the connected components of  $M(u_1,\dots,u_N)$ are precisely $F_1,\dots,F_l$, and therefore $c(M(u_1,\dots,u_N))=l$.

	First, we prove that $F_j$ is connected, for $j\in\{1,\dots,l\}$. In fact, the  unit outer normal vector set of $P_j$ in the subspace $\xi_j$ is $\{u_i,-u_i:i\in F_j\}$. By the fact that $P_j$ is directly indecomposable and Case 1,
	it follows that $F_j$ is connected.
	
	Second, we prove that for distinct $j_1,j_2\in\{1,\dots,l\}$, $i_1$ is not equivalent to $i_2$, for $i_1\in F_{j_1}$ and $i_2\in F_{j_2}.$
	Otherwise, assume there exist $ i_1\in F_{j_1}$ and $i_2\in F_{j_2}$ such that $i_1$ is equivalent to $i_2$. By the definition of equivalence, there exists a circuit $C$ so that $\{i_1,i_2\}\subseteq C$. Since a circuit is a minimal dependent set, it follows that $C\setminus F_{j_1}$ and $C\cap F_{j_1}$ are both independent sets. Since $\mathrm{span}\{u_i:i\in C\setminus F_{j_1}\}$ is orthogonal to $\xi_{j_1}$, it follows that $C=(C\setminus F_{j_1})\cup (C\cap F_{j_1})$ is an independent set, which contradicts to that $C$ is a dependent set.

 Consequently,  $c(M(u_1,\dots,u_N))=l$.
\end{proof}

\vskip10pt
\section{\bf Proofs of Theorems  1.1 and 1.2}\label{sec5}
\vskip5pt

Now, with the aid of Theorem \ref{vertxedge}, we prove our main results. We first prove Theorem \ref{thm1.2}.
\begin{theorem}\label{fX_3}
	Let $P$ be a polytope in $\mathbb{R}^n$ with its centroid at the origin and $n\geq5$. Then
	$$\frac{X_3(P)}{V_n(P)}\leq\sqrt[n]{\binom{n}{3}\big((\frac{3}{n})^n-3(\frac{2}{n})^n+\frac{3}{n^n}\big)}$$
	with equality if and only if $P$ is a parallelotope.
\end{theorem}

\begin{proof}
	Assume that $\mathrm{supp}S_P\cup\mathrm{supp}S_{-P}=\{\pm u_1,\pm u_2,\dots,\pm u_N\}$ and $V_n(P)=n.$
	
	If $P$ is a  parallelotope, by Example \ref{X_3}, it follows that
	\[\frac{X_3(P)^n}{V_n(P)^n}=\binom{n}{3}\big((\frac{3}{n})^n-3\big((\frac{2}{n})^n-2\frac{1}{n^n}\big)-3\frac{1}{n^n}\big)=\binom{n}{3}\big((\frac{3}{n})^n-3(\frac{2}{n})^n+\frac{3}{n^n}\big).\]
	In the following, we aim to show that
	\[\frac{X_3(P)^n}{V_n(P)^n}<\binom{n}{3}\big((\frac{3}{n})^n-3(\frac{2}{n})^n+\frac{3}{n^n}\big),\]
	as long as  $P$ is \emph{not} a parallelotope.
	
	Let
	$$\{\xi_1^2,\xi_2^2,\dots,\xi_{m_2}^2\}=\{\mathrm{span}\{u_{i_1},\dots,u_{i_n}\}:i_1,\dots,i_n\in\{1,2,\dots,N\},\dim(\mathrm{span}\{u_{i_1},\dots,u_{i_n}\})=2\};$$
	$$\{\xi_1^3,\xi_2^3,\dots,\xi_{m_3}^3\}=\{\mathrm{span}\{u_{i_1},\dots,u_{i_n}\}:i_1,\dots,i_n\in\{1,2,\dots,N\},\dim(\mathrm{span}\{u_{i_1},\dots,u_{i_n}\})=3\}.$$
	That is, $\{\xi_1^2,\xi_2^2,\dots,\xi_{m_2}^2\}$ consists of \emph{all} the $2$-dimensional subspaces spanned
	by $n$  outer  normals of $P$; $\{\xi_1^3,\xi_2^3,\dots,\xi_{m_3}^3\}$ consists of \emph{all} the $3$-dimensional subspaces spanned
	by $n$  outer  normals of $P$.
	
	Let $a=(a_1,\dots,a_N)=(V_P(\{\pm u_1\}),\dots,V_P(\{\pm u_N\}))$. From   Example \ref{X_3},
	it follows that \[X_3(P)^n=\sum\limits_{i=1}^{m_3}\Big( (\sum\limits_{\{j:u_j\in \xi_i^3\}}a_j)^n-\sum\limits_{\{l:\xi_l^2\subseteq \xi_i^3\}}\big( (\sum\limits_{\{j:u_j\in \xi_l^2\}}a_j)^n-\sum\limits_{\{j:u_j\in \xi_l^2\}} a_j^n\big) -\sum\limits_{\{j:u_j\in \xi_i^3\}}a_j^n\Big).\]
	
	In the rest, we investigate the maximum of the functional
	\[g(x_1,\dots,x_N)=\sum\limits_{i=1}^{m_3}\Big( (\sum\limits_{\{j:u_j\in \xi_i^3\}}x_j)^n-
	\sum\limits_{\{l:\xi_l^2\subseteq \xi_i^3\}}\big((\sum\limits_{\{j:u_j\in \xi_l^2\}}x_j)^n-
	\sum\limits_{\{j:u_j\in \xi_l^2\}} x_j^n\big) -\sum\limits_{\{j:u_j\in \xi_i^3\}}x_j^n\Big)\]
	in the matroid polytope $P_{M(u_1,\dots,u_N)}$. Obviously that $g(a)=X_3(P)^n$.
	
	Since $g$ is a polynomial and $P_{M(u_1,\dots,u_N)}$ is compact, it follows that $g$ attains its maximum on $P_{M(u_1,\dots,u_N)}$.
	What follows aims to show that $g$ attains its maximum precisely at  \emph{vertices} of the matroid polytope $P_{M(u_1,\dots,u_N)}$. We show this by contradiction.
	
	\textbf{Step 1.} Assume that $g$ attains its maximum at $x\in P_{M(u_1,\dots,u_N)}\setminus\mathcal{F}_0(P_{M(u_1,\dots,u_N)})$.
	By Theorem \ref{vertxedge}, there exist distinct $j_1,j_2\in\{1,\dots,N\}$ and sufficiently small
	$\varepsilon>0$ so that
	\[x+\varepsilon[e_{j_2}-e_{j_1},e_{j_1}-e_{j_2}]\subseteq P_{M(u_1,\dots,u_N)}.\]
	W.l.o.g., assume $j_1=1,\ j_2=2$.
	
	Let $G(t)=g(x+t(e_1-e_2))$, $t\in[-\varepsilon,\varepsilon]$.
	By the assumption that $g$  attains its maximum at the point $x$, it implies that $G(t)$ attains its maximum at $t=0$.
	Hence, it is necessary that
	$$G'(0)=0,\ \ \ \ \text{and}\ \ \ \ G''(0)=(\frac{\partial^2g}{\partial x_1^2}+\frac{\partial^2g}{\partial x_2^2}-2\frac{\partial^2g}{\partial x_1\partial x_2})(x)\le0.$$
	
	However, in the following we show that $\frac{\partial^2g}{\partial x_{1}^2}-\frac{\partial^2g}{\partial x_{1}\partial x_{2}}>0$ (Similarly, that $\frac{\partial^2g}{\partial x_{2}^2}-\frac{\partial^2g}{\partial x_{1}\partial x_{2}}>0$.), which is a contradiction. To prove it, we divide it into four sub-steps.
	
	\textbf{Step 1.1.} To calculate and simplify $\frac{\partial^2g}{\partial x_{1}^2}-\frac{\partial^2g}{\partial x_{1}\partial x_{2}}$.
	By calculating directly, we obtain
	\begin{align*}
		&\frac{\partial g}{\partial x_{1}}=n\sum\limits_{\{i:u_1\in\xi_i^3\}}\Big((\sum_{\{j:u_j\in\xi_i^3\}}x_j)^{n-1}-\sum\limits_{\{l:u_1\in\xi_l^2\subseteq \xi_i^3\}}\big((\sum_{\{j:u_j\in\xi_l^2\}}x_j)^{n-1}-x_1^{n-1}\big) -x_1^{n-1}\Big),\\
		&\frac{\partial^2g}{\partial x_{1}^2}=n(n-1)\sum\limits_{\{i:u_1\in\xi_i^3\}}\Big((\sum_{\{j:u_j\in\xi_i^3\}}x_j)^{n-2}-\sum\limits_{\{l:u_1\in\xi_l^2\subseteq \xi_i^3\}}\big( (\sum_{\{j:u_j\in\xi_l^2\}}x_j)^{n-2}-x_1^{n-2}\big) -x_1^{n-2}\Big),\\
		&\frac{\partial^2g}{\partial x_{1}\partial x_2}=n(n-1)\sum\limits_{\{i:u_1,u_2\in\xi_i^3\}}\big( (\sum_{\{j:u_j\in\xi_i^3\}}x_j)^{n-2}-(\sum_{\{j:u_j\in\xi_{u_1,u_2}\}}x_j)^{n-2}\big),
	\end{align*}
	where $\xi_{u_1,u_2}=\mathrm{span}\{u_1,u_2\}\in\{\xi_1^2,\xi_2^2,\dots,\xi_{m_2}^2\}$. Hence,
	\begin{align*}
		&\frac{1}{n(n-1)}(\frac{\partial^2g}{\partial x_{1}^2}-\frac{\partial^2g}{\partial x_{1}\partial x_{2}})\\
		=&\sum\limits_{\{i:u_{1}\in\xi^3_i,u_{2}\notin\xi^3_i\}}\Big( (\sum_{\{j:u_j\in\xi_i^3\}}x_j)^{n-2}-\sum\limits_{\{l:u_1\in\xi_l^2\subseteq \xi_i^3\}}\big((\sum_{\{j:u_j\in\xi_l^2\}}x_j)^{n-2}-x_1^{n-2}\big) -x_1^{n-2}\Big)\\
		&-\sum\limits_{\{i:u_{1},u_{2}\in\xi^3_i\}} \sum\limits_{\{l:u_{1}\in\xi^2_l\subseteq\xi^3_i,u_2\notin\xi_l^2\}}\big( (\sum_{\{j:u_j\in\xi_l^2\}}x_j)^{n-2}-x_{1}^{n-2}\big)\\
		\triangleq&A-B.
	\end{align*}	
	
	To further calculate A and B, w.l.o.g., assume~$\{\xi_i^2:u_{1}\in\xi^2_i\}=\{\xi^2_{1},\xi^2_2,\dots,\xi^2_{k}\}$, $k\leq m_2$, and $\xi_1^2=\xi_{u_1,u_2}$. Two observations are in order.
	
	First, $u_1,u_2\in\mathrm{span}\{ u_{2},\xi_{l}^2\}\in\{\xi_1^3,\xi_2^3,\dots,\xi_{m_3}^3\}$, $l=2,3,\dots,k$.
	
	Second, for distinct $i_1,i_2\in\{i:u_{1},u_{2}\in\xi^3_i\}$, $\xi_{i_1}^3\cap\xi_{i_2}^3=\xi_{u_1,u_2}$.
	
	Combining the above two observations, it follows that
	\begin{equation}\label{2sum}
		\bigcup_{i\in \{i:u_{1},u_{2}\in\xi^3_i\}}\{l:u_{1}\in\xi^2_l\subseteq\xi^3_i,~\text{and} ~u_2\notin\xi_l^2\}
	\end{equation}
	is indeed a \emph{disjoint} union of  $\{2,3,\dots,k\}$. Hence,
	\[B=\sum\limits_{\{i:u_{1},u_{2}\in\xi^3_i\}} \sum\limits_{\{l:u_{1}\in\xi^2_l\subseteq\xi^3_i,u_2\notin\xi_l^2\}}\big( (\sum_{\{j:u_j\in\xi_l^2\}}x_j)^{n-2}-x_{1}^{n-2}\big)=\sum\limits_{l=2}^k\big( (\sum_{\{j:u_j\in\xi_l^2\}}x_j)^{n-2}-x_{1}^{n-2}\big).\]
	
	Let
	\[y_l=\sum_{\{j:u_j\in\xi_l^2\setminus\{u_1\}\}}x_j=-x_1+\sum_{\{j:u_j\in\xi_l^2\}}x_j,\quad l=1,2,\dots,k.\]	
	Then $y_l\le\sum_{\{j:u_j\in\xi_l^2\}}x_j\le2$, $B=\sum\limits_{l=2}^k\big( (x_{1}+y_l)^{n-2}-x_{1}^{n-2}\big)$ and
	\[A=\sum\limits_{\{i:u_{1}\in\xi^3_i,u_{2}\notin\xi^3_i\}}\Big((\sum_{\{j:u_j\in\xi_i^3\}}x_j)^{n-2}-\sum\limits_{\{l:u_1\in\xi_l^2\subseteq \xi_i^3\}}\big( (x_1+y_l)^{n-2}-x_1^{n-2}\big) -x_1^{n-2}\Big).\]
	
	Observe that for each $i\in\{i:u_1\in\xi_i^3,u_{2}\notin\xi_i^3\},$
	\[\bigcup_{l\in\{l:\xi_{l}^2\subseteq\xi^3_i\}}\{j:u_j\in\xi_l^2\setminus\{u_1\}\}\]
	is indeed a \emph{disjoint} union of $\{j\in\{2,3,\dots,N\}:u_j\in\xi_i^3\setminus\{u_1\}\}$. We obtain that
	\begin{equation}\label{3sum}
		\sum_{\{j:u_j\in\xi_i^3\}}x_j=x_1+\sum_{\{j:u_j\in\xi_i^3\setminus\{u_1\}\}}x_j=x_1+\sum\limits_{\{l:u_1\in\xi_{l}^2\subseteq\xi^3_i\}}y_l,\quad \forall i\in \{i:u_{1}\in\xi^3_i,u_{2}\notin\xi_i^3\}.
	\end{equation}
	Therefore,
	\[A=\sum\limits_{\{i:u_{1}\in\xi^3_i,u_{2}\notin\xi^3_i\}} \Big((x_{1}+\sum\limits_{\{l:u_1\in\xi_{l}^2\subseteq\xi^3_i\}}y_l)^{n-2} - \sum\limits_{\{l:u_1\in\xi_{l}^2\subseteq\xi^3_{i}\}}\big((x_{1}+y_l)^{n-2}- x_{1}^{n-2}\big)-x_{1}^{n-2}\Big).\]

	\textbf{Step 1.2.} Varying $x_1$ in $A-B$, let
	\begin{align*}
		h_\alpha(t)=&\sum\limits_{\{i:u_{1}\in\xi^3_i,u_{2}\notin\xi^3_i\}} \Big((t+\sum\limits_{\{l:u_1\in\xi_{l}^2\subseteq\xi^3_i\}}y_l)^{\alpha} -\sum\limits_{\{l:u_1\in\xi_{l}^2\subseteq\xi^3_{i}\}}\big( (t+y_l)^{\alpha}- t^{\alpha}\big)-t^{\alpha}\Big)\\
		&-\sum\limits_{l=2}^k\big( (t+y_l)^{\alpha}-t^{\alpha}\big),\quad\quad\quad t\in[0, \infty),\quad \alpha=3,4,\dots,n-2.
	\end{align*}
	Thus, $A-B=h_{n-2}(x_1)$, and
	\[\frac{d^{n-2-\alpha }}{(dt)^{n-2-\alpha}}h_{n-2}=\frac{(n-2)!}{\alpha !}h_{\alpha},\quad\alpha=3,4,\dots,n-3.\]
	By the mean value theorem, there exists an~$\eta\in (0,1)$ satisfying
	\begin{align*}
		h_{n-2}(x_{1})=&h_{n-2}(0)+(n-2)h_{n-3}(0)x_{1}+(n-2)(n-3)h_{n-4}(0)\frac{x_{1}^2}{2!}+\cdots\\
		&+\binom{n-2}{n-6}h_{4}(0)x_{1}^{n-6}+\binom{n-2}{n-5}h_{3}(\eta x_{1})x_{1}^{n-5}.
	\end{align*}

	If we can show that $h_{n-2}(x_{1})>0,$ then $A-B>0$ is derived.
	
	\textbf{Step 1.3.} To show that $h_\alpha(0)>0$, $\alpha=4,5,\dots,n-2$.
	
	By $\alpha\ge 4$ and collecting the terms with the factor $\alpha y_l^{\alpha-1}$, we obtain
	\begin{align*}
		h_\alpha(0)=&\sum\limits_{\{i:u_{1}\in\xi^3_i,u_{2}\notin\xi^3_i\}} \big\{(\sum\limits_{\{l:u_1\in\xi_{l}^2\subseteq\xi^3_i\}}y_l)^{\alpha} - \sum
		_{\{l:u_1\in\xi_{l}^2\subseteq\xi^3_{i}\}}y_l^{\alpha}\big\}-\sum_{l=2}^ky_l^{\alpha}\\
		\ge&\sum\limits_{\{i:u_{1}\in\xi^3_i,u_{2}\notin\xi^3_i\}} \big\{\sum_{\{l:u_1\in\xi_{l}^2\subseteq\xi^3_i\}}\alpha y_l^{\alpha-1}\sum\limits_{\{s:u_1\in\xi_{s}^2\subseteq\xi^3_i,s\neq l\}}y_s \big\}-\sum\limits_{l=2}^ky_l^{\alpha}\\
		=&\sum\limits_{l=2}^k\alpha y_l^{\alpha-1}\sum\limits_{\big\{s:u_1\in\xi_s^2,u_{2}\notin\mathrm{span}\{\xi^2_{s},\xi^2_{l}\},s\neq l\big\}}y_s-\sum_{l=2}^ky_l^{\alpha}\\
		=&\sum_{l=2}^k\alpha y_l^{\alpha-1}\sum_{\big\{s:u_1\in\xi_s^2,\xi^2_{s}\nsubseteq\mathrm{span}\{u_{2},\xi^2_{l}\}\big\}}y_s-\sum_{l=2}^ky_l^{\alpha}.
	\end{align*}
	
	Note that
	$\bigcup_{l\in\{l:u_{1}\in\xi_l^2\}}\{j:u_j\in\xi_l^2\setminus\{ u_1\}\}$
	is indeed a \emph{disjoint} union of $\{2,3,\dots,N\}$. We obtain that
	\begin{equation}\label{sumxi}
		\sum\limits_{l=1}^ky_l=\sum\limits_{\{l:u_{1}\in\xi_l^2\}}\sum_{\{j:u_j\in\xi_l^2\setminus\{ u_1\}\}}x_j=\sum_{j=2}^{N}x_j=n-x_1.
	\end{equation}
	
	From (\ref{sumxi}), (\ref{3sum}), (\ref{prez}), and that $y_l\le2$, $\alpha\ge 4, n\ge 5$, it follows that
	\begin{align*}
		h_\alpha(0)\ge&\sum\limits_{l=2}^k\alpha y_l^{\alpha-1}\big( n-x_{1}-\sum\limits_{\big\{s:u_1\in\xi^2_{s}\subseteq\mathrm{span}\{u_{2},\xi^2_{l}\}\big\}}\!\!y_s\big)-\sum\limits_{l=2}^ky_l^{\alpha}\\
		=&\sum\limits_{l=2}^k\alpha y_l^{\alpha-1}(n-x_1-(-x_1+\sum_{\{j:u_j\in\mathrm{span}\{u_{2},\xi^2_{l}\}\}}\!\!x_j))-\sum\limits_{l=2}^ky_l^{\alpha}\\
		=&\sum\limits_{l=2}^k\alpha y_l^{\alpha-1}(n-\sum_{\{j:u_j\in\mathrm{span}\{u_{2},\xi^2_{l}\}\}\!\!}x_j)-\sum\limits_{l=2}^ky_l^{\alpha}\\
		=&\sum\limits_{l=2}^ky_l^{\alpha-1}(n\alpha -\alpha\sum_{\{j:u_j\in\mathrm{span}\{u_{2},\xi^2_{l}\}\}}\!\!x_j-y_l)\\
		\ge&\sum\limits_{l=2}^ky_l^{\alpha-1}(n\alpha-3\alpha-2)>0.
	\end{align*}
	
	\textbf{Step 1.4.} To show that $h_3(t)>0$, for $t\in(0,x_{1})$.
	
	By collecting the terms with the factor $3 y_l^{2}$ and $y_l$, respectively, (\ref{sumxi}) and (\ref{3sum}), we obtain
	\begin{align*}
		&h_3(t)\\
		=&\sum\limits_{\{i:u_{1}\in\xi^3_i,u_{2}\notin\xi^3_i\}}\!\!\!\Big((t+\sum\limits_{\{l:u_{1}\in\xi_{l}^2\subseteq\xi^3_i\}}\!\!\!y_l)^{3}-
		\sum\limits_{\{l:u_{1}\in\xi_{l}^2\subseteq\xi^3_{i}\}}\!\!\!\big( (y_l+t)^{3}- t^{3}\big)-t^{3}\Big)-\sum\limits_{l=2}^k\big( (t+y_l)^{3}-t^{3}\big)\\
		\ge&\sum\limits_{\{i:u_{1}\in\xi^3_i,u_{2}\notin\xi^3_i\}}\!\!\!
		\big(\sum\limits_{\{l:u_{1}\in\xi_{l}^2\subseteq\xi^3_i\}}\!\!\!3y_l^2\sum\limits_{\{s:u_{1}\in\xi_{s}^2\subseteq\xi^3_i,s\neq l\}}\!\!\!y_s +3t\!\!\! \sum\limits_{\{l:u_{1}\in\xi_{l}^2\subseteq\xi^3_i\}}\!\!\!y_l\sum\limits_{\{s:u_{1}\in\xi_{s}^2\subseteq\xi^3_i,s\neq l\}}\!\!\!y_s\big)-\sum\limits_{l=2}^k\big( (t+y_l)^{3}-t^{3}\big)\\ =&\sum\limits_{l=2}^k3y_l^2\sum\limits_{\big\{s:u_{1}\in\xi_s^2,\xi^2_{s}\nsubseteq\mathrm{span}\{u_{2},\xi^2_{l}\}\big\}}\!\!\!y_s+3t\sum\limits_{l=2}^ky_l\sum\limits_{\big\{s:u_{1}\in\xi_s^2,\xi^2_{s}\nsubseteq\mathrm{span}\{u_{2},\xi^2_{l}\}\big\}}\!\!\!y_s-\sum\limits_{l=2}^k(y_l^3+3y_l^2t+3y_lt^2)\\
		=&\sum\limits_{l=2}^k3y_l^2(n-\sum_{\{j:u_j\in\mathrm{span}\{u_2,\xi_l^2\}\}}\!\!\!x_j)+3t\sum\limits_{l=2}^ky_l(n-\sum_{\{j:u_j\in\mathrm{span}\{u_2,\xi_l^2\}\}}\!\!\!x_j)-\sum\limits_{l=2}^k(y_l^3+3y_l^2t+3y_lt^2)\\
		\ge&\sum\limits_{l=2}^k3y_l^2(n-3)+3t\sum\limits_{l=2}^ky_l(n-3)-\sum\limits_{l=2}^k(y_l^3+3y_l^2t+3y_lt^2)\\
		=&\sum\limits_{l=2}^ky_l^2(3n-9-y_l-3t)+3t\sum\limits_{l=2}^ky_l(n-3-t).
	\end{align*}
	
	Since $0<t<x_1\le1$, $y_l\le2$, and $n\ge 5$, we have
	\[h_3(t)>\sum\limits_{l=2}^ky_l^2(3n-9-2-3)+3t\sum\limits_{l=2}^ky_l(n-3-1)=\sum\limits_{l=2}^ky_l^2(3n-14)+3t\sum\limits_{l=2}^ky_l(n-4)>0.\]
	
	Combining the above four sub-steps, we have shown that
	\[\frac{\partial^2g}{\partial x_{j_1}^2}-\frac{\partial^2g}{\partial x_{j_1}\partial x_{j_2}}=n(n-1)(A-B)={n(n-1)}h_{n-2}(x_{1})>0.\]	
	Consequently, $g$ cannot attain its maximum in the region $P_{M(u_1,\dots,u_N)}\setminus\mathcal{F}_0(P_{M(u_1,\dots,u_N)})$.
	
	\textbf{Step 2.} To evaluate the functional $g$ at points of $\mathcal{F}_0(P_{M(u_1,\dots,u_N)})$.
	
	For each $x\in\mathcal{F}_0(P_{M(u_1,\dots,u_N)})$, by Theorem \ref{vertxedge}, w.l.o.g., assume
	\[x=(\underbrace{1,\dots,1}_n,\underbrace{0,\dots,0}_{N-n}).\]
	
	By the definition of $P_{M(u_1,\dots,u_N)}$, it follows that $u_1, u_2,\dots,u_n$ must be linearly independent. So, for each $i\in\{1,2,\dots,m_3\}$, there are at most three elements of $\{u_1,u_2,\dots,u_n\}$
	lying in $\xi_i^3$. Moreover, if the number of elements of $\{u_1,u_2,\dots,u_n\}$
	lying in $\xi_i^3$ is strictly less than 3, i.e., $|\xi_i^3\cap\{u_1,\dots,u_n\}|<3$, then
	\[(\sum\limits_{\{j:u_j\in \xi_i^3\}}x_j)^n-\sum\limits_{\{l:\xi_l^2\subseteq \xi_i^3\}}\big((\sum\limits_{\{j:u_j\in \xi_l^2\}}x_j)^n-\sum\limits_{\{j:u_j\in \xi_l^2\}} x_j^n\big) -\sum\limits_{\{j:u_j\in \xi_i^3\}}x_j^n=0\]
	by calculating directly when $|\xi_i^3\cap\{u_1,\dots,u_n\}|=0,1,2$.
	
	Therefore,
	\begin{align*}
		g(x)=&\sum\limits_{i=1}^{m_3}\Big((\sum\limits_{\{j:u_j\in \xi_i^3\}}x_j)^n-\sum\limits_{\{l:\xi_l^2\subseteq \xi_i^3\}}\big((\sum\limits_{\{j:u_j\in \xi_l^2\}}x_j)^n-\sum\limits_{\{j:u_j\in \xi_l^2\}} x_j^n\big) -\sum\limits_{\{j:u_j\in \xi_i^3\}}x_j^n\Big)\\
		=&\sum\limits_{\{i:|\xi_i^3\cap\{u_1,\dots,u_n\}|=3\}}\Big( (\sum\limits_{\{j:u_j\in \xi_i^3\}}x_j)^n-\sum\limits_{\{l:\xi_l^2\subseteq \xi_i^3\}}\big((\sum\limits_{\{j:u_j\in \xi_l^2\}}x_j)^n-\sum\limits_{\{j:u_j\in \xi_l^2\}} x_j^n\big) -\sum\limits_{\{j:u_j\in \xi_i^3\}}x_j^n\Big)\\
		=&\binom{n}{3}\big(3^n-3(2^n-2)-3\big)\\
		=&\binom{n}{3}\big(3^n-3\times 2^n+3\big).
	\end{align*}
	
	\textbf{Step 3.} Combining Step 1 and Step 2, we obtain that for $x\in P_{M(u_1,\dots,u_N)}$,
	\[g(x)\le\binom{n}{3}\big(3^n-3\times 2^n+3\big), \]
	with equality if and only if $x\in\mathcal{F}_0(P_{M(u_1,\dots,u_N)})$.
	
	If $P$ is \emph{not} a parallelotope, then  $N>n$. By Theorem \ref{rel=cal}, we obtain that
	$$a=(V_P(\{\pm u_1\}),\dots,V_P(\{\pm u_N\}))\in\mathrm{relint}P_{M(u_1,\dots,u_N)}.$$
	Hence,
	\[X_3(P)^n=g(a)<\max_{x\in P_{M(u_1,\dots,u_N)}} g(x)=\binom{n}{3}\big(3^n-3\times 2^n+3\big).\]
	
	To sum up, for a polytope $P$ with its centroid at the origin, it follows that	
	$$\frac{X_3(P)}{V_n(P)}\leq\sqrt[n]{\binom{n}{3}\big((\frac{3}{n})^n-3(\frac{2}{n})^n+\frac{3}{n^n}\big)},$$
	with equality if and only if $P$ is a parallelotope.
\end{proof}

Now, we prove Theorem \ref{thm1.1}.
\begin{theorem}\label{pX_2}
	Let $P$ be a polytope in $\mathbb{R}^n$ with its centroid at the origin and $n\geq3$. Then
	$$\frac{X_2(P)}{V_n(P)}\leq\sqrt[n]{\binom{n}{2}\big((\frac{2}{n})^n-\frac{2}{n^n}\big)},$$
	with equality if and only if $P$ is a parallelotope.
\end{theorem}

\begin{proof}
	Assume that $\mathrm{supp}S_P\cup\mathrm{supp}S_{-P}=\{\pm u_1,\pm u_2,\dots,\pm u_N\}$ and $V_n(P)=n.$
	
	If $P$ is a  parallelotope, then by Example \ref{X_2}, it follows that
	\[X_2(P)^n=\binom{n}{2}\big(2^n-2\big).\]
	In the following, we aim to show that
	\[X_2(P)^n<\binom{n}{2}\big(2^n-2\big),\]
	as long as  $P$ is \emph{not} a parallelotope.
	
	Let
	$$\{\xi_1^2,\xi_2^2,\dots,\xi_{m_2}^2\}=\{\mathrm{span}\{u_{i_1},\dots,u_{i_n}\}:i_1,\dots,i_n\in\{1,2,\dots,N\},\dim(\mathrm{span}\{u_{i_1},\dots,u_{i_n}\})=2\}.$$
	That is, $\{\xi_1^2,\xi_2^2,\dots,\xi_{m_2}^2\}$ consists of \emph{all} the $2$-dimensional subspaces spanned
	by $n$ unit outer  normals of the polytope $P$.
	
	Let $a=(a_1,\dots,a_N)=(V_P(\{\pm u_1\}),\dots,V_P(\{\pm u_N\}))$. From  the assumption that $V_n(P)=n$ and Example \ref{X_2}, it follows that \[\frac{X_2(P)^n}{V_n(P)^n}=\sum\limits_{i=1}^{m_2}\big((\sum\limits_{\{j:{u_j}\in\xi^2_{i}\}}a_j)^n-\sum\limits_{\{j:{u_j}\in\xi^2_{i}\}}a_j^n\big).\]
	
	In the rest, we investigate the maximum of the functional
	\[f(x_1,\dots,x_N)=\sum\limits_{i=1}^{m_2}\big((\sum\limits_{\{j:{u_j}\in\xi^2_{i}\}}x_j)^n-\sum\limits_{\{j:{u_j}\in\xi^2_{i}\}}x_j^n\big)\]
	in the set $P_{M(u_1,\dots,u_N)}$.
	Obviously that $f(a)=X_2(P)^n/V_n(P)^n$.
	
	Since $P_{M(u_1,\dots,u_N)}$ is compact and $f$ is a polynomial, it follows that $f$ attains its maximum on $P_{M(u_1,\dots,u_N)}$.
	What follows aims to show that $f$ attains its maximum precisely at the \emph{vertices} of the polytope $P_{M(u_1,\dots,u_N)}$. We show this by contradiction.
	
	\textbf{Step 1.} Assume that $f$ attains its maximum at $x\in P_{M(u_1,\dots,u_N)}\setminus\mathcal{F}_0(P_{M(u_1,\dots,u_N)})$. By Theorem \ref{vertxedge}, there exist distinct $j_1,j_2\in\{1,\dots,N\}$ and sufficiently small
	$\varepsilon>0$ so that
	\[x+\varepsilon[e_{j_2}-e_{j_1},e_{j_1}-e_{j_2}]\subseteq P_{M(u_1,\dots,u_N)}.\]
	W.l.o.g., assume $j_1=1,\ j_2=2$.
	
	Let
	\[F(t)=f(x+t(e_1-e_2)),\quad t\in[-\varepsilon,\varepsilon].\]
	By the assumption that $f$  attains its maximum at the point $x$, it implies that $F(t)$ attains its maximum at $t=0$.
	Hence, it is necessary that
	$$F'(0)=0,\ \ \ \ \text{and}\ \ \ \ F''(0)=(\frac{\partial^2f}{\partial x_1^2}+\frac{\partial^2f}{\partial x_2^2}-2\frac{\partial^2f}{\partial x_1\partial x_2})(x)\le0.$$
	
	However, in the following we show that $\frac{\partial^2f}{\partial x_{1}^2}-\frac{\partial^2f}{\partial x_{1}\partial x_{2}}>0$
	(Similarly, that $\frac{\partial^2f}{\partial x_{2}^2}-\frac{\partial^2f}{\partial x_{1}\partial x_{2}}>0$.), which is a contradiction. To prove it, we divide it into two sub-steps.
	
	\textbf{Step 1.1.} By calculating directly, we obtain
	\begin{align*}
		&\frac{\partial f}{\partial x_{1}}=n\sum\limits_{\{i:u_{1}\in\xi_i^2\}}\big( (\sum_{\{j:u_j\in\xi_i^2\}}x_j)^{n-1}-x_{1}^{n-1}\big),\\
		&\frac{\partial^2f}{\partial x_{1}^2}=n(n-1)\sum\limits_{\{i:u_{1}\in\xi_i^2\}}\big( (\sum_{\{j:u_j\in\xi_i^2\}}x_j)^{n-2}-x_{1}^{n-2}\big)\big),\\
		&\frac{\partial^2f}{\partial x_{1}\partial x_{2}}=n(n-1)(\sum_{\{j:u_j\in\xi_{u_1,u_2}\}}x_j)^{n-2},
	\end{align*}
	where $\xi_{u_1,u_2}=\mathrm{span}\{u_1,u_2\}\in\{\xi_1^2,\xi_2^2,\dots,\xi_{m_2}^2\}$.
	
	Since $n\ge 3$, it follows that
	\begin{align}
		&~~~~\nonumber\frac{1}{n(n-1)}\big(\frac{\partial^2f}{\partial x_{1}^2}-\frac{\partial^2f}{\partial x_{1}\partial x_{2}}\big)\\
		&\nonumber=\sum\limits_{\{i:u_{1}\in\xi_i^2\}}\big( (\sum_{\{j:u_j\in\xi_i^2\}}x_j)^{n-2}-x_{1}^{n-2}\big)- \big(\sum_{\{j:u_j\in\xi_{u_1,u_2}\}}x_j\big)^{n-2}\\
		\nonumber&=-x_1^{n-2}+\sum\limits_{\{i:u_{1}\in\xi_i^2,u_2\notin\xi_i^2\}}\big( (\sum_{\{j:u_j\in\xi_i^2\}}x_j)^{n-2}-x_{1}^{n-2}\big)\\
		\nonumber&=-x_1^{n-2}+\sum\limits_{\{i:u_{1}\in\xi_i^2,u_2\notin\xi_i^2\}}\big( (x_1+\sum_{\{j:u_j\in\xi_i^2\setminus\{ u_1\}\}}x_j)^{n-2}-x_{1}^{n-2}\big)\\
		\nonumber&\ge-x_1^{n-2}+\sum\limits_{\{i:u_{1}\in\xi_i^2,u_2\notin\xi_i^2\}}\big( x_1^{n-2}+(n-2)x_1^{n-3}(\sum_{\{j:u_j\in\xi_i^2\setminus\{ u_1\}\}}x_j)-x_{1}^{n-2}\big)\\
		\label{-u1u2}&=-x_1^{n-2}+(n-2)x_1^{n-3}\sum\limits_{\{i:u_{1}\in\xi_i^2,u_2\notin\xi_i^2\}}\sum_{\{j:u_j\in\xi_i^2\setminus\{u_1\}\}}x_j.
	\end{align}
	
	\textbf{Step 1.2.}
	Since  $x\pm\varepsilon(e_1-e_2)=(x_1\pm\varepsilon,x_2\mp\varepsilon,x_3,\dots,x_N)\in P_{M(u_1,\dots,u_N)}$, it follows that
	\begin{equation}\label{5.5}
		0<x_1<1.
	\end{equation}
	
	Combining (\ref{-u1u2}), (\ref{sumxi}), (\ref{5.5}) and the definition of $P_{M(u_1,\dots,u_N)}$, we obtain that
	\begin{align*}
		&~~~~\frac{1}{n(n-1)}\big(\frac{\partial^2f}{\partial x_{1}^2}-\frac{\partial^2f}{\partial x_{1}\partial x_{2}}\big)\\
		&\ge-x_1^{n-2}+(n-2)x_1^{n-3}(n-x_1-\sum_{\{j:u_j\in\xi_{u_1,u_2}\setminus\{u_1\}\}}x_j)\\
		&=-x_1^{n-2}+(n-2)x_1^{n-3}(n-x_1-\sum_{\{j:u_j\in\xi_{u_1,u_2}\}}x_j+x_1)\\
		&=x_1^{n-3}\big(-x_1+(n-2)(n-\sum_{\{j:u_j\in\xi_{u_1,u_2}\}}x_j)\big)\\
		&> x_1^{n-3}\big(-1+(n-2)(n-2)\big)\\
		&=(n-1)(n-3)x_1^{n-3}\ge 0.
	\end{align*}
	Consequently, $f$ cannot attain its maximum in the region $P_{M(u_1,\dots,u_N)}\setminus\mathcal{F}_0(P_{M(u_1,\dots,u_N)})$.
	
	\textbf{Step 2.} To evaluate the functional $f$ at points of $\mathcal{F}_0(P_{M(u_1,\dots,u_N)})$.
	
	For each $x\in\mathcal{F}_0(P_{M(u_1,\dots,u_N)})$, by Theorem \ref{vertxedge}, w.l.o.g., assume
	\[x=(\underbrace{1,\dots,1}_n,\underbrace{0,\dots,0}_{N-n}).\]
	
	By the definition of $P_{M(u_1,\dots,u_N)}$, it follows that $u_1, u_2,\dots,u_n$ must be linearly independent. Thus, for each $i\in\{1,2,\dots,m_2\}$, there are at most two elements of $\{u_1,u_2,\dots,u_n\}$
	lying in $\xi_i^2$. Moreover, if the number of elements of $\{u_1,u_2,\dots,u_n\}$
	lying in $\xi_i^2$ is strictly less than 2, i.e., $|\xi_i^2\cap\{u_1,\dots,u_n\}|<2$, then
	\[\big(\sum\limits_{\{j:{u_j}\in\xi^2_{i}\}}x_j\big)^n-\sum\limits_{\{j:{u_j}\in\xi^2_{i}\}}x_j^n=0\]
	by calculating directly when $|\xi_i^2\cap\{u_1,\dots,u_n\}|=0,1$.
	
	Therefore,
	\begin{align*}
		f(x)=&\sum\limits_{i=1}^{m_2}\big((\sum\limits_{\{j:{u_j}\in\xi^2_{i}\}}x_j)^n-\sum\limits_{\{j:{u_j}\in\xi^2_{i}\}}x_j^n\big)\\
		=&\sum\limits_{\{i:|\xi_i^2\cap\{u_1,\dots,u_n\}|=2\}}\big((\sum\limits_{\{j:{u_j}\in\xi^2_{i}\}}x_j)^n-\sum\limits_{\{j:{u_j}\in\xi^2_{i}\}}x_j^n\big)\\
		=&\binom{n}{2}\big(2^n-2\big).
	\end{align*}

	\textbf{Step 3.} Combining Step 1 and Step 2, we obtain that for $x\in P_{M(u_1,\dots,u_N)}$,
	\[f(x)\le\binom{n}{2}\big(2^n-2\big), \]
	with equality if and only if $x\in\mathcal{F}_0(P_{M(u_1,\dots,u_N)})$.
	
	If $P$ is \emph{not} a parallelotope, then  $N>n$. By Theorem \ref{rel=cal}, we obtain that
	$$a=(V_P(\{\pm u_1\}),\dots,V_P(\{\pm u_N\}))\in\mathrm{relint}P_{M(u_1,\dots,u_N)}.$$
	Hence,
	\[X_2(P)^n=f(a)<\max_{x\in P_{M(u_1,\dots,u_N)}} f(x)=\binom{n}{2}\big(2^n-2\big).\]
	
	To sum up, for a polytope $P$ with its centroid at the origin, it follows that	
	$$\frac{X_2(P)}{V_n(P)}\leq\sqrt[n]{\binom{n}{2}\big((\frac{2}{n})^n-\frac{2}{n^n}\big)},$$
	with equality if and only if $P$ is a parallelotope.
\end{proof}

\vskip 10pt
\section{\bf Proof of Theorem 1.3}\label{sec6}
\vskip5pt

Let $P\in \mathcal{P}_4^4$ with its centroid at the origin. Assume that
$$V_4(P)=4,\quad\text{and}\quad\mathrm{supp}S_P\cup\mathrm{supp}S_{-P}=\{\pm u_1,\dots,\pm u_N\}.$$

Recall that $\mathcal{P}_4^4$ is the set of polytopes in $\mathbb{R}^4$  which are in $4$-general position and contain the origin in their interiors. In this setting,  any $4$ elements of $\{u_1,\dots,u_N\}$ are linearly independent. So, the bases $\mathcal{B}(M(u_1,\dots,u_N))$ is the collection of  all $4$-element subsets of $\{ 1,\dots,N\}$, and therefore $P_{M(u_1,\dots,u_N)}$ is the hypersimplex $\bigtriangleup_N^4=[0,1]^4\cap H_{\mathrm{I}_N,4}$.

Write $a=(a_1,\dots,a_N)=(V_P(\{\pm u_1\}),\dots,V_P(\{\pm u_N\})).$  By Example \ref{X_3}, it follows that
\[\frac{X_3(P)^4}{V_4(P)^4}=\frac{1}{256}\sum\limits_{1\leq i<j<k\leq N}\big((a_i+a_j+a_k)^4-(a_i+a_j)^4-(a_i+a_k)^4-(a_j+a_k)^4+a_i^4+a_j^4+a_k^4\big).\]

For $x\in[0,1]^4\cap H_{\mathrm{I}_N,4}$, let
\[f(x)=\frac{1}{256}\sum\limits_{1\leq i<j<k\leq N}\big((x_i+x_j+x_k)^4-(x_i+x_j)^4-(x_i+x_k)^4-(x_j+x_k)^4+x_i^4+x_j^4+x_k^4\big). \]
Obviously that $f$ is \emph{symmetric} in $x_i$, $i\in{1,2,\dots,N}$,  and $f(a)=X_3(P)^4/V_4(P)^4$.

In the following, we study the extremal problem for $X_3$ by two cases of quantity $N$, i.e., the number of outer normal vector pairs of $P$. To prove Theorem \ref{thm1.3}, we first show two lemmas.

\begin{lemma}\label{lemN5}
	If $N=5$, then
	$$ f(a)\leq \frac{72}{125},$$
	with equality if and only if $a=\frac{4}{5}(1,1,1,1,1).$
\end{lemma}

\begin{proof}
	
	If $a=\frac{4}{5}(1,1,1,1,1),$
	then
	$f(a)=\binom{5}{3}\big((\frac{3}{5})^4-3(\frac{2}{5})^4+3(\frac{4}{5})^4\big)=\frac{72}{125}.$
	
	Since $f$ is a polynomial, it follows that $f$ attains its maximum on $[0,1]^5\cap H_{(1,1,1,1,1),4}$. It suffices to show that if $a\neq\frac{4}{5}(1,1,1,1,1)$,
	then $f$ \emph{cannot} attain its maximum at $a$.
	
	Otherwise, assume $f$ attains its maximum at
	$$a\in [0,1]^5\cap H_{(1,1,1,1,1),4}\setminus\{\frac{4}{5}(1,1,1,1,1)\},$$
	w.l.o.g., say $a=(a_1,a_2,a_3,a_4,a_5)$ satisfying that
	$$0\leq a_1\leq a_2 \leq a_3\leq a_4\leq a_5\leq 1.$$
	
	Since $a\neq\frac{4}{5}(1,1,1,1,1)$, it yields that $0\leq a_1<a_5\leq 1$. Let $\varepsilon=\min\{a_5,1-a_1\}$. Then $\varepsilon>0$ and
	$$a+te_1-te_5\in [0,1]^5\cap H_{(1,1,1,1,1),4},\quad t\in[0,\varepsilon] .$$
	
	Let $F(t)=f(a+te_1-te_5),\ t\in[0,\varepsilon].$
	Since $f$ attains its maximum at $a$, it follows that $F(t)$ attains
	its maximum at $t = 0$. Hence, it is necessary that $F_+^{\prime}(0)\leq 0.$ However, we will show that $F_+^{\prime}(0)>0$, which is a contradiction.
	
	By directly calculating, we obtain
	\begin{align*}
		\frac{\partial f}{\partial x_1}&=\frac{1}{64}\sum\limits_{2\leq j<k\leq 5}\big((x_1+x_j+x_k)^3-(x_1+x_j)^3-(x_1+x_k)^3+x_1^3\big)\\
		&=\frac{3}{64}\sum\limits_{2\leq j<k\leq 5}(2x_1x_jx_k+x_j^2x_k+x_jx_k^2)\\
		&=\frac{3}{64}\sum\limits_{2\leq j\le 4}(2x_1x_jx_5+x_j^2x_5+x_jx_5^2)+\frac{3}{64}\sum\limits_{2\leq i<j\leq 4}(2x_1x_ix_j+x_i^2x_j+x_ix_j^2);\\
		\frac{\partial f}{\partial x_5}&=\frac{3}{64}\sum\limits_{2\leq j\le 4}(2x_1x_jx_5+x_j^2x_1+x_jx_1^2)+\frac{3}{64}\sum\limits_{2\leq i<j\leq 4}(2x_5x_ix_j+x_i^2x_j+x_ix_j^2).
	\end{align*}
	
	Since $0\le a_1\le a_2\le a_3 \le a_4 \le a_5$, and $\sum_{j=1}^5a_j=4$, it follows that
	\begin{align*}
		F_+^{\prime}(0)=&(\frac{\partial f}{\partial x_1}-\frac{\partial f}{\partial x_5})(a)\\
		=&\frac{3}{64}(a_5-a_1)\sum\limits_{2\leq j\leq 4}a_j(a_1+a_5+a_j)+\frac{3}{32}(a_1-a_5)\sum\limits_{2\leq i<j\leq 4} a_ia_j\\
		=&\frac{3}{64}(a_5-a_1)\sum\limits_{2\leq j\leq 4}a_j(a_1+a_5+a_j)-\frac{3}{32}(a_5-a_1)\frac{1}{2}\sum\limits_{2\leq i,j\leq 4,i\neq j} a_ia_j\\
		=&\frac{3}{64}(a_5-a_1)\sum_{j=2}^4a_j(a_1+a_5+a_j)-\frac{3}{64}(a_5-a_1)\sum_{j=2}^4 a_j(a_2+a_3+a_4-a_j)\\
		=&\frac{3}{64}(a_5-a_1)\sum_{j=2}^4 a_j(a_1+a_5+2a_j-a_2-a_3-a_4)\\
		\ge&\frac{3}{64}(a_5-a_1)\sum_{j=2}^4 a_2(a_1+a_5+2a_j-a_2-a_3-a_4)\\
		=& \frac{3}{64}(a_5-a_1)a_2(3a_1+3a_5-a_2-a_3-a_4)\ge 0,
	\end{align*}
	with equality if and only if $a_1=0$ and $a_2=a_3=a_4=a_5=1$.
	
	However,
	\[f(0,1,1,1,1)=\binom{4}{3}\big((\frac{3}{4}\big)^4-3(\frac{2}{4})^4+3(\frac{1}{4})^4\big)=\frac{9}{16}=\frac{1125}{2000}<\frac{1152}{2000}=\frac{72}{125}=f(\frac{4}{5},\frac{4}{5},\frac{4}{5},\frac{4}{5},\frac{4}{5}),\]
	which implies that $f$ \emph{cannot} attain its maximum at $(0,1,1,1,1)$. So, $a\neq(0,1,1,1,1)$, and therefore $F_+^{\prime}(0)>0$, which contradicts to that $F_+^{\prime}(0)\le0$.
\end{proof}

\begin{lemma}\label{lemN6}
	If $N\ge 6$ and $x\in (0,1]^N\cap H_{\mathrm{I}_N,4},$  then $x$ is not the maximal point of $f$ in $[0,1]^N\cap H_{\mathrm{I}_N,4}$.
\end{lemma}

\begin{proof}
	Assume $f$ attains its maximum at $x\in (0,1]^N\cap H_{\mathrm{I}_N,4}.$
	W.o.l.g., say $$0<x_1\leq x_2\leq\cdots\leq x_N\leq 1.$$
	
	\textbf{Step 1.} We prove that there exists a sufficiently small  $\varepsilon>0$, so that
	$$x+t(e_1-e_2)\in[0,1]^N\cap H_{\mathrm{I}_N,4},\quad t\in[-\varepsilon,\varepsilon].$$
	
	If $x_2=1$, then $x_3=\cdots=x_N=1$. However, by $N\ge 6$, it yields that
	$$ 4=\sum_{i=1}^Nx_i>\sum_{i=3}^Nx_i=N-2\ge 4,$$
	which is a contradiction. So, $x_2<1$. Let
	$$\varepsilon=\min\{x_1,x_2,1-x_1,1-x_2\}.$$
	Then $\varepsilon>0$ and satisfies the above inclusion.

	\textbf{Step 2.} Let $F(t)=f(x+te_1-te_2),\ t\in[-\varepsilon,\varepsilon].$
	By the assumption that $f$ attains its maximum at $x$, it yields that $F(t)$ attains
	its maximum at $t = 0$. Hence, it is necessary that
	$$
	F^{\prime}(0)=0, \quad \text { and } \quad F^{\prime \prime}(0)=\big(\frac{\partial^{2} f}{\partial x_{1}^{2}}+\frac{\partial^{2} f}{\partial x_{2}^{2}}-2 \frac{\partial^{2} f}{\partial x_{1} \partial x_{2}}\big)(x) \leq 0.
	$$
	In the following, it suffices to show that $F^{\prime \prime}(0) > 0$ to get a contradiction.
	
	\textbf{Step 2.1.} we calculate $\frac{\partial^{2} f}{\partial x_{1}^{2}}, \frac{\partial^{2} f}{\partial x_{2}^{2}}$ and $\frac{\partial^{2} f}{\partial x_{1} \partial x_{2}}$.
	\begin{align*}
		\frac{\partial f}{\partial x_1}(x)=&\frac{1}{64}\sum_{2\leq j<k\leq N}\big((x_1+x_j+x_k)^{3}-(x_1+x_j)^3-(x_1+x_k)^3+x_1^3\big)\\
		=&\frac{1}{64}\sum_{k=3}^N\big((x_1+x_2+x_k)^3-(x_1+x_2)^3-(x_1+x_k)^3+x_1^3\big)+\frac{3}{64}\sum\limits_{3\leq j<k\leq N}x_jx_k(x_j+x_k+2x_1);\\
		\frac{\partial f}{\partial x_2}(x)=&\frac{1}{64}\sum_{k=3}^N\big( (x_1+x_2+x_k)^3-(x_1+x_2)^3-(x_2+x_k)^3+x_2^3\big)+\frac{3}{64}\sum_{3\leq j<k\leq N}x_jx_k(x_j+x_k+2x_2).
	\end{align*}
	
	Thus,
	\begin{align*}
		\frac{\partial^2 f}{\partial x_1^2}(x)=&\frac{3}{64}\sum_{k=3}^N\big( (x_1+x_2+x_k)^2-(x_1+x_2)^2-x_k(2x_1+x_k)\big)+\frac{3}{32}\sum\limits_{3\leq j<k\leq N}x_jx_k;~\\
		\frac{\partial^2 f}{\partial x_2^2}(x)=&\frac{3}{64}\sum_{k=3}^N\big( (x_1+x_2+x_k)^2-(x_1+x_2)^2-x_k(2x_2+x_k)\big)+\frac{3}{32}\sum\limits_{3\leq j<k\leq N}x_jx_k;\\
		\frac{\partial^2 f}{\partial x_1x_2}(x)=&\frac{3}{64}\sum_{k=3}^N\big( (x_1+x_2+x_k)^2-(x_1+x_2)^2\big).
	\end{align*}
	
	\textbf{Step 2.2.} Since $\sum_{j=1}^Nx_j=4$ and that $x_1\le x_2\le \cdots \le x_{N-1}\le x_N$, it follows that
	\begin{align}
		F^{\prime \prime}(0)&=\big(\frac{\partial^{2} f}{\partial x_{1}^{2}}+\frac{\partial^{2} f}{\partial x_{2}^{2}}-2 \frac{\partial^{2} f}{\partial x_{1} \partial x_{2}}\big)(x)\nonumber\\
		&= \frac{3}{16}\sum_{3\le j< k \le N}x_jx_k-\frac{3}{64}\sum_{k=3}^Nx_k(2x_1+2x_2+2x_k)\nonumber\\
		&=\frac{3}{16}\times\frac{1}{2}\sum_{3\le j,k \le N,j\neq k}x_jx_k-\frac{3}{32}\sum_{k=3}^Nx_k(x_1+x_2+x_k)\nonumber\\
		&=\frac{3}{32}\sum_{k=3}^Nx_k\sum_{3\le j\le N,j\neq k}x_j-\frac{3}{32}\sum_{k=3}^Nx_k(x_1+x_2+x_k)\nonumber\\
		&=\frac{3}{32}\sum_{k=3}^Nx_k(4-x_1-x_2-x_k)-\frac{3}{32}\sum_{k=3}^Nx_k(x_1+x_2+x_k)\nonumber\\
		&=\frac{3}{32}\sum_{k=3}^Nx_k(4-2x_1-2x_2-2x_k)\nonumber\\
		&\ge \frac{3}{32}\big(x_N(4-2x_1-2x_2-2x_N)+\sum_{k=3}^{N-1}x_k(4-2x_1-2x_2-2x_{N-1})\big)\label{6.1},
	\end{align}
	with equality if and only if $x_3=x_4=\cdots=x_{N-1}$.
	
	Since $N\ge 6$, it follows that
	$$2x_1+2x_2+x_{N-1}+x_N\le x_1+x_2+x_3+x_4+x_{N-1}+x_N\le 4.$$
	Hence,
	\begin{equation}\label{N=6}
		4-2x_1-2x_2\ge x_N+x_{N-1},
	\end{equation}
	with equality if and only if $N=6,$ and $x_1=x_2=x_3=x_4$.
	
	Combining (\ref{6.1}), (\ref{N=6}), that $\sum_{k=1}^Nx_k=4$, and that $x_1,x_2<1$ and $x_N\le1$, we obtain
	\begin{align}
		F^{\prime \prime}(0)&\ge \frac{3}{32}\big(x_N(x_N+x_{N-1}-2x_N)+\sum_{k=3}^{N-1}x_k(x_N+x_{N-1}-2x_{N-1})\big)\nonumber\\
		&=\frac{3}{32}\big(-x_N(x_N-x_{N-1})+(x_N-x_{N-1})\sum_{k=3}^{N-1}x_k\big)\nonumber\\
		&=\frac{3}{32}(x_N-x_{N-1})(4-x_1-x_2-x_N-x_N)\ge 0\label{6.3},
	\end{align}
	where the last equality holds if and only if $x_N=x_{N-1}$.
	
	Combining (\ref{6.1}) (\ref{N=6}) and (\ref{6.3}), it concludes that $$F^{\prime \prime}(0)\ge 0,$$
	with equality if and only if $N=6$, and $x_1=x_2=x_3=x_4=x_5=x_6=\frac{4}{6}$.
	
	However,
	\[f(\frac{4}{6},\frac{4}{6},\frac{4}{6},\frac{4}{6},\frac{4}{6},\frac{4}{6})=\binom{6}{3}\big((\frac{3}{6})^4-3(\frac{2}{6})^4+3(\frac{1}{6})^4\big)=\frac{5}{9}=\frac{625}{1125}<\frac{648}{1125}=\frac{72}{125}=f(\frac{4}{5},\frac{4}{5},\frac{4}{5},\frac{4}{5},\frac{4}{5},0),\]
	which implies $f$ \emph{cannot} attain its maximum at $\frac{4}{6}(1,1,1,1,1,1)$. So, $x\neq\frac{4}{6}(1,1,1,1,1,1)$,
	and therefore $F''(0)>0$, which contradicts to that $F''(0)\le0$.
	
	Consequently, if  $x\in (0,1]^N\cap H_{\mathrm{I}_N,4}$, then $x$ is not the maximal point of $f$ in $[0,1]^N\cap H_{\mathrm{I}_N,4}$.
\end{proof}

Combining Lemma \ref{lemN5} and Lemma \ref{lemN6}, we immediately obtain the following result.
\begin{theorem}
	\label{x3p44}
	Let $N\ge 5$ and $x\in [0,1]^N\cap H_{\mathrm{I}_N,4}.$ Then
	$$f(x)\leq \frac{72}{125},$$
	with equality if and only if there exists $\{i_1,i_2,i_3,i_4,i_5\}\subseteq\{1,2,\dots,N\}$, so that $x_{i_j}=\frac{4}{5}$, $j=1,2,3,4,5.$
\end{theorem}

\begin{proof}
	Since $f$ is a polynomial, it follows that $f$ attains its maximum on $[0,1]^N\cap H_{\mathrm{I}_N,4}$. Assume $f$ attains its maximum at $\bar{x}$.
	Since $f$ is symmetric in $x_i$, $i\in{1,2,\dots,N}$, w.l.o.g., say
	\[\bar{x}=(\bar{x}_1,\bar{x}_2,\dots,\bar{x}_M,0,\dots,0),\quad \bar{x}_1,\bar{x}_2,\dots,\bar{x}_M>0,\ 4\le M\le N.\]
	
	Since
	\[f(1,1,1,1,0,\dots,0)<f(\frac{4}{5},\frac{4}{5},\frac{4}{5},\frac{4}{5},\frac{4}{5},0,\dots,0),\]
	it follows that $M\ge 5$.
	
	For $x\in[0,1]^M\cap H_{\mathrm{I}_M,4}$, let
	\[f_M(x)=\frac{1}{256}\sum\limits_{1\leq i<j<k\leq M}\big( (x_i+x_j+x_k)^4-(x_i+x_j)^4-(x_i+x_k)^4-(x_j+x_k)^4+x_i^4+x_j^4+x_k^4\big). \]
	Then $f_M(\bar{x}_1,\bar{x}_2,\dots,\bar{x}_M)=f(\bar{x})$. Note that
	\[f_M(\bar{x}_1,\bar{x}_2,\dots,\bar{x}_M)\le\max_{x\in[0,1]^M\cap H_{\mathrm{I}_M,4}}f_M(x)\le\max_{x\in[0,1]^N\cap H_{\mathrm{I}_N,4}}f(x)=f(\bar{x}).\]
	Hence, $f_M$ attains its maximum at $(\bar{x}_1,\bar{x}_2,\dots,\bar{x}_M)$. Combining Lemma \ref{lemN6} and Lemma \ref{lemN5}, it follows that $M=5$
	and $\bar{x}_1=\bar{x}_2=\bar{x}_3=\bar{x}_4=\bar{x}_5=\frac{4}{5}.$
\end{proof}

Now, we finish the proof of Theorem \ref{thm1.3}.

\begin{theorem}
	Let $P\in \mathcal{P}_4^4$ with its centriod at the origin. Then
	\[ \frac{X_3(P)}{V_4(P)}\leq\sqrt[4]{\frac{72}{125}},\]
	with equality if and only if $\mathrm{supp}S_P\cup\mathrm{supp}S_{-P}=\{\pm u_1,\dots,\pm u_5\}$,
	and $V_P(\{\pm u_i\})=\frac{V_4(P)}{5},$ $i=1,\dots,5.$
\end{theorem}

\begin{proof}
	Assume that $\mathrm{supp}S_P\cup\mathrm{supp}S_{-P}=\{\pm u_1,\pm u_2,\dots,\pm u_N\}$ and $V_4(P)=4.$
	
	If $P$ is a parallelotope, then by Example \ref{X_3}, it follows that
	$$\frac{X_3(P)^4}{V_4(P)^4}=\binom{4}{3}\big((\frac{3}{4})^4-3(\frac{2}{4})^4+3(\frac{1}{4})^4\big)=\frac{9}{16}<\frac{72}{125}.$$
	
	If $P$ is \emph{not} a parallelotope, then $N\ge 5$. Let $a=(V_P(\{\pm u_1\}),\dots,V_P(\{\pm u_N\}))$. From Theorem \ref{x3p44}, we obtain
	\[\frac{X_3(P)^4}{V_4(P)^4}=f(a)\leq\max_{x\in[0,1]^N\cap H_{\mathrm{I}_N,4}} f(x)=\frac{72}{125},\]
	with equality if and only if $N=5$ and $V_P(\{\pm u_i\})=\frac{4}{5},\ i=1,2,3,4,5.$
\end{proof}

In particular, if $P$ is a \emph{simplex} in $\mathbb{R}^4$ with its centroid at the origin, then the equality holds in the above theorem. Moreover, by Lemma \ref{thm5.1} and Lemma \ref{injective}, there exists an origin-symmetric \emph{decahedron} $Q\in\mathcal{P}_4^4$, so that the equality also holds. However, according to Example \ref{exm3.5}, it is striking that $\frac{X_3}{V_4}$ does \emph{not} attain its extremum at the simplex or decahedron $Q$ in $\mathcal{P}_c^4$.

	\vskip 25pt
	\bibliographystyle{amsplain}
	
\end{document}